\documentclass{article}
\usepackage[utf8]{inputenc} 
\usepackage{amsmath}
\usepackage{amsfonts}
\usepackage{amssymb}
\usepackage{graphicx}
\usepackage{hyperref}
\usepackage{authblk}
\usepackage{color}%
\usepackage{esint}
\usepackage{enumitem}
\usepackage{mathtools}
\usepackage[title,titletoc]{appendix}
\usepackage{xpatch,amsthm}
\makeatletter
\xpatchcmd{\@thm}{\fontseries\mddefault\upshape}{}{}{} 
\makeatother
\providecommand{\U}[1]{\protect\rule{.1in}{.1in}}
\newtheorem{theorem}{Theorem}

\newtheorem{claim}[theorem]{Claim}

\newtheorem{definition}[theorem]{Definition}

\newtheorem{lemma}[theorem]{Lemma}
\newtheorem{notation}[theorem]{Notation}

\newtheorem{proposition}[theorem]{Proposition}
\newtheorem{remark}[theorem]{Remark}

\begin{document}
	 \author{Swarnendu Sil\thanks{swarnendu.sil@fim.math.ethz.ch}}
	 \affil{Forschungsinstitut für Mathematik \\ ETH Z\"{u}rich}
	\title{Topology of weak $G$-bundles via Coulomb gauges in critical dimensions}
	\maketitle 

\begin{abstract}
The transition maps for a Sobolev $G$-bundle are not continuous in the critical dimension and thus the usual notion of topology does not make sense. In this work, we show that if such a bundle $P$ is equipped with a Sobolev connection $A$, then one can associate a topological isomorphism class to the pair $\left( P, A\right),$ which is invariant under Sobolev gauge changes and coincides with the usual notions for regular bundles and connections. This is based on a regularity result which says any bundle in the critical dimension in which a Sobolev connection is in Coulomb gauges are actually $C^{0,\alpha}$ for any $\alpha < 1.$ We also show any such pair can be strongly approximated by smooth connections on smooth bundles. Finally, we prove that for  sequences $(P^{\nu},A^{\nu})$ with uniformly bounded $n/2$-Yang-Mills energy, the topology stabilizes if the $n/2$ norm of the curvatures are equiintegrable. This implies a criterion to detect topological flatness in Sobolev bundles in critical dimensions via $n/2$-Yang-Mills energy.  \end{abstract}
\section{Introduction}
Throughout this article, we shall assume that $n \geq 3, N \geq 1$ are integers and 
\begin{itemize}
	\item $k=1$ or $2$ and $2 < p < \infty$ is a real number, 
	\item $G$ is a compact finite dimensional Lie group,
	 \item $M^{n}$ is a connected, closed $n$-dimensional smooth Riemannian manifold. 
\end{itemize} 
Here closed means a compact manifold without boundary. We are concerned with principal $G$-bundles over $M^{n}.$ The analysis for Yang-Mills functional and in general, problems related to higher dimensional gauge theory, often requires one to work with notions of Sobolev principal bundles and Sobolev connections on them, where the connection forms and the transition maps, which define \u{C}ech cocycles, are only $W^{k,p}.$ But since the transition maps need not be continuous if $kp \leq n,$ the notion of topological isomorphism classes of bundles no longer make sense. 

\par One of our goal in this article is to show that in the critical dimension $kp =n,$ one can however associate a unique topological isomorphism class to a pair $\left( P, A \right)$, where $A$ is a connection on $P$ such that $A \in L^{n}$ and $dA \in L^{\frac{n}{2}}.$ Our notion of a topological isomorphism class is assigned to the pair $(P, A)$ and \emph{not} to $P$ alone. This explicit dependence on the connection $A$ is a new point of view in which  we are encoding topological information about the bundle in the connection as well, so that analysis at the level of connections can still keep track of topological information about the underlying bundles. We fully expect this new point of view to be more useful than the usual topological notions in critical and supercritical regime, since in this regime, the connections are not constrained to respect the topology of the bundles and can `drag' the bundles along with them. The case of supercritical dimensions however requires other tools, which will be treated in a forthcoming work \cite{Sil_MorreySobolevInPrep}. 

\par The topological isomorphism class is nothing but the $C^{0}$-equivalence class of the corresponding  Coulomb bundle, i.e. the bundle obtained from $P$ by a $W^{k,p}$ gauge change in which the connection $A$ satisfies the Coulomb condition $d^{\ast}A = 0.$ As we shall show, given the pair $\left(P, A \right),$ any corresponding Coulomb bundle has the same  $C^{0}$-equivalence class. The fact that such bundles are $C^{0}$-bundles has been proved by Rivi\`{e}re \cite{Riviere_QuantizationYangMills}. We shall show a stronger result, that these bundles are actually H\"{o}lder continuous with any H\"{o}lder exponent $\alpha < 1.$

\par This assignment of $C^{0}$-equivalence class to a pair $(P, A)$ is stable under $W^{k,p}$ gauge changes for $kp=n$ and if the connection and the bundle are more regular, this notion coincides with the usual notion of topological isomorphism class for bundles and thus would be independent of the connection. The H\"{o}lder continuity of the Coulomb bundles is already noticed by Shevchishin in \cite{Shevchishin_limitholonomyYangMills}, although it does not seem to be widely known. Much like our approach, Shevchishin is also using this improved regularity to implicitly define a notion of topology for bundles in the critical dimension. However, instead of assigning a topology to the \emph{pair} $(P, A),$ he is assigning the topology to the bundle $P$ \emph{alone}, by implicitly making a specific choice for the connection $A.$ But two different $\mathcal{U}^{k,p}$ connections can give rise to two distinct Coulomb bundles which are not $C^{0}$ isomorphic ( see Remark \ref{specific choice of connection} ) and there is little geometric reason to prefer any one connection over another.  

\par As a by product, we prove that in the critical dimension, any Sobolev cocycle can be approximated arbitrarily closely in the strong Sobolev topology by smooth cocycles, up to passing to a refinement of the cover.  
\begin{theorem}[cocycle smoothing]\label{cocycle smoothing}
	Let $\left\lbrace U_{\alpha}\right\rbrace_{\alpha \in I} $ be a good cover of $M^{n}$ and let $\left\lbrace g_{\alpha\beta} \right\rbrace_{\alpha, \beta \in I}$ be a collection of maps such that $g_{\alpha\beta} \in W^{k,p}\left( U_{\alpha}\cap U_{\beta}; G\right)$, with $kp=n,$ for every $\alpha, \beta \in I$ with $U_{\alpha}\cap U_{\beta} \neq \emptyset,$ $g_{\alpha\alpha}=\mathbf{1}_{G}$ for every $\alpha \in I$ and satisfies the cocycle conditions
	\begin{align}
	g_{\alpha\beta}(x)g_{\beta\gamma}(x)= g_{\alpha\gamma}(x) \qquad \text{ for a.e. } x \in U_{\alpha}\cap U_{\beta}\cap U_{\gamma} 
	\end{align} 
	for every $\alpha, \beta, \gamma \in I$ with $U_{\alpha}\cap U_{\beta}\cap U_{\gamma} \neq \emptyset.$ Then given any $ \varepsilon >0,$ there exists a good refinement $\left\lbrace V_{j}\right\rbrace_{j \in J}$ of $\left\lbrace U_{\alpha}\right\rbrace_{\alpha \in I}$ and \textbf{smooth} maps $g^{\varepsilon}_{ij}\in C^{\infty}\left(V_{i}\cap V_{j}; G\right)  $ for all $i,j \in J$ with $V_{i}\cap V_{j} \neq \emptyset,$ satisfying $g^{\varepsilon}_{jj}=\mathbf{1}_{G}$ for every $j \in J$ and 
	$$g^{\varepsilon}_{ij}(x)g^{\varepsilon}_{jk}(x) = g^{\varepsilon}_{ik}(x) \quad \text{ for a.e. } x \in V_{i}\cap V_{j}\cap V_{k},  $$
		 whenever $V_{i}\cap V_{j}\cap V_{k} \neq \emptyset$ and we have  $$ \left\lVert g^{\varepsilon}_{ij} - g_{\phi(i)\phi(j)}\right\rVert_{W^{k,p}\left( V_{i}\cap V_{j};G\right)} \leq \varepsilon \qquad \text{ whenever } V_{i}\cap V_{j} \neq \emptyset,$$ where $\phi: J \rightarrow I$ is the refinement map. 
\end{theorem}
This answers a question raised by Rivi\`{e}re in \cite{Riviere_Yang_Millsnote_2015arXiv}. We deduce this theorem from the more general Theorem \ref{smoothaapproxcriticalconn}, which says roughly the following:\bigskip 

\emph{For $kp=n,$ given a $W^{k,p}$ principal $G$-bundle over $M^{n}$ equipped with a $\mathcal{U}^{k,p}$ connection, for any $\varepsilon>0$ we can find a smooth principal $G$-bundle over $M^{n}$ equipped with a smooth connection so that the bundle is both $W^{k,p}$ equivalent and $\varepsilon$-close to the original bundle in $W^{k,p}$ norm and the connection is also $\varepsilon$-close to the pullback of the original connection in $\mathcal{U}^{k,p}$ norm.}\bigskip

Similar results are proved in Isobe \cite{Isobe_Sobolevbundlecriticaldimension}. But our proof is different and follows a more connection oriented approach, which also highlights the fact that the topology of smooth bundles defined by the approximating smooth cocycles are not necessarily uniquely determined by the original cocycles alone ( see Remark \ref{possiblydifftopology_smoothapprox} ).  Also, since our analysis is based on Coulomb gauges, which always exist by linear Hodge theory when $G$ is Abelian, we prove in Theorem \ref{Abeliansmoothaapproxsupercritical} that such an approximation is possible even in supercritical dimensions for principal $S^{1}$-bundles. This result is also proved by Isobe in a different article \cite{Isobe_SobolevbundleAbelian}. But our proof is not only different, but also considerably easier in the Abelian case. 

\par The benefits of encoding topological information in the bundle-connection pair bear fruits in the analysis of sequences of bundles with connections under uniformly bounded Yang-Mills energy. In Theorem \ref{equiintegrability theorem}, which is the main result of the article, we show that for a sequence of pairs $\left( P^{\nu}, A^{\nu}\right)$ with uniformly bounded $n/2$-Yang-Mills energy, the associated topological classes, which can a priori be all different, stabilize for large enough $\nu$ if the sequence of $n/2$ norm of the curvatures is equiintegrable. For a sequence of connections on a fixed $W^{k,p}\cap C^{0}$ bundle, this yields the following. 
\begin{theorem}[Stability of topology if curvatures does not concentrate]\label{improved strong convergence}
	Let $kp=n.$ Let $P$ be a $W^{k,p}\cap C^{0}$ bundle over $M^{n}$  and let $\left\lbrace A^{\nu}\right\rbrace_{\nu \geq 1}$ be a sequence of connections on $P$ such that 
	\begin{itemize}
		\item[(1)] $A^{\nu} \in L^{n}$, $dA^{\nu} \in L^{\frac{n}{2}},$ $d^{\ast}A^{\nu} \in L^{\left(\frac{n}{2},1\right)}$ for every $\nu \geq 1, $ 
		\item[(2)] $\left\lVert F_{A^{\nu}}\right\rVert_{L^{\frac{n}{2}}\left(M^{n}; \Lambda^{2}T^{\ast}M^{n}\otimes \mathfrak{g} \right)}$ is uniformly bounded, 
		\item[(3)] the sequence $\left\lbrace \left\lvert F_{A^{\nu}}\right\rvert^{\frac{n}{2}} \right\rbrace_{\nu \geq 1}$ is \emph{equiintegrable} in $M^{n}.$
	\end{itemize}
Then there exists a subsequence $\left\lbrace A^{\nu_{s}}\right\rbrace_{s \geq 1},$ a limiting $W^{k,p}\cap C^{0}$ bundle $P^{\infty} = \left( \left\lbrace U^{\infty}_{i}\right\rbrace_{i \in I}, \left\lbrace g_{ij}^{\infty}\right\rbrace_{i,j \in I} \right) $ with $\left[ P\right]_{C^{0}} = \left[ P^{\infty}\right]_{C^{0}} $ and a limit connection $A^{\infty}$ on $P^{\infty}$  such that for every $ i \in I$,
	\begin{align*}
	F_{A^{\nu_{s}}_{i}} &\rightharpoonup F_{A^{\infty}_{i}} &&\text{ weakly in } L^{\frac{n}{2}}\left(U^{\infty}_{i}; \Lambda^{2}T^{\ast}U^{\infty}_{i}\otimes\mathfrak{g}\right).
	\end{align*} \end{theorem}  
 This improves Theorem IV.2. in Rivi\`{e}re \cite{Riviere_QuantizationYangMills}, which needed $A^{\nu}$ to be strongly convergent in $W^{1,\frac{n}{2}}$ and $d^{\ast}A^{\nu}$ to be strongly convergent in the Lorentz space $L^{(\frac{n}{2}, 1)}.$ Control of the full gradient of the connection, control of  $d^{\ast}A^{\nu}$ and the strong convergences, are all somewhat unnatural and unsatisfactory requirements. In comparison, Theorem \ref{improved strong convergence} does not need either strong convergences and except for the information on the curvatures, not even an uniform bound is needed either for $d^{\ast}A^{\nu}$ or the full gradient, which settles the question raised by Rivi\`{e}re in Remark IV. 2. in \cite{Riviere_QuantizationYangMills}. Theorem \ref{equiintegrability theorem} implies Theorem \ref{flatness in W1n bundles}, which gives a criterion to detect \emph{topological flatness} for $W^{k,p}$ bundles equipped with $\mathcal{U}^{k,p}$ connections via $n/2$ Yang-Mills energy of the connection for $kp=n.$ As a consequence, we deduce the following extension of the energy gap theorem to non-smooth connections, which as far as we are aware, is new and might be of interest in itself.
 \begin{theorem}[$n/2$-Yang-Mills energy gap]\label{energy gap}
	For any cover $\mathcal{U}$ of $M^{n},$ there exists a constant $\delta >0,$ depending only on $\mathcal{U}$, $M^{n}$ and $G$ such that if $ P $ is a $W^{1,n}\cap C^{0}$ bundle trivialized over $\mathcal{U}$ and $A$ is a connection form on $P$ such that  $A \in L^{n}$, $dA \in L^{\frac{n}{2}},$ $d^{\ast}A \in L^{\left(\frac{n}{2},1\right)},$ then we must have $YM_{n/2}\left(A \right)  > \delta , $  unless $P$ is flat.  
\end{theorem} 
\noindent When $A$ is a smooth, this is the usual energy gap theorem.
\par The requirement of equiintegrability of the $n/2$-norm of the curvatures in Theorem \ref{improved strong convergence}, which at first sight might seem strange, is actually a natural hypothesis. In practice, if we know that the sequence of curvatures satisfy some elliptic systems, for example, in cases of stationary Yang-Mills or Anti-Self-Dual connections etc, then by the epsilon-regularity type results for elliptic systems in critical dimensions, the curvatures does not concentrate in the so-called \emph{neck regions} and the equiintegrability hypothesis is satisfied on \emph{such regions} and would be satisfied on the whole domain if there are no \emph{bubbles}. On the other hand, it is known ( see Freed-Uhlenbeck \cite{FreedUhlenbeck_instantonandfourmanifold}, Taubes \cite{Taubes_metricconnectiongluing} )  that the topology can change in the weak limit if one assumes only the uniform $L^{\frac{n}{2}}$ bound of the curvatures. 

\paragraph*{} The rest of the article is organized as follows. In Section \ref{prelim}, we collect the preliminary notions and notations that we would use. Section \ref{elliptic estimates density} we are concerned with proving the smooth approximation theorems. Section \ref{topology section} defines the notion of the topological isomorphism class and discusses its properties and proves the result concerning topology stabilization in the limit and its consequences. 
 
\section{Preliminaries}\label{prelim}
\subsection{Smooth principal $G$ bundles with connections}
A smooth principal $G$-bundle ( or simply a $G$-bundle ) $P$ over $M^{n}$ is usually denoted by the notation $P \stackrel{\pi}{\rightarrow} M^{n},$ where $\pi: P \rightarrow M^{n}$ is a smooth map, called the \emph{projection map}, $P$ is called the \emph{total space} of the bundle, $M^{n}$ is the \emph{base space}. One way to define a smooth principal $G$ bundle $P \stackrel{\pi}{\rightarrow} M^{n},$ is to specify an open cover $\mathcal{U} = \left\lbrace U_{\alpha} \right\rbrace_{\alpha \in I} $ of $M^{n},$ i.e. $M^{n}= \bigcup\limits_{\alpha \in I} U_{\alpha}$ and a collection of bundle trivialization maps $\left\lbrace \phi_{\alpha}\right\rbrace_{\alpha \in I}$ such that $\phi_{\alpha}: U_{\alpha}\times G \rightarrow \pi^{-1}\left( U_{\alpha} \right) $ is a smooth diffeomorphism for every $\alpha \in I$ and each of them preserves the fiber, i.e. $\pi \left(\phi_{\alpha}\left(x, g\right)\right) = x$ for every $g \in G$ for every $x \in U_{\alpha}$ and they are $G$-equivariant, i.e. whenever $U_{\alpha}\cap U_{\beta}$ is nonempty, there exist smooth maps, called transition function $g_{\alpha\beta}: U_{\alpha}\cap U_{\beta} \rightarrow G$ such that for every $x \in U_{\alpha}\cap U_{\beta},$ we have 
\begin{align}\label{Gequivariance}
\left( \phi^{-1}_{\alpha} \circ \phi_{\beta}\right) (x,h) = (x,g_{\alpha\beta}(x)h ) \text{ for every } h \in G.
\end{align}  
  From \eqref{Gequivariance}, it is clear that $g_{\alpha\alpha}=\mathbf{1}_{G},$ the identity element of $G,$ for all $\alpha \in I$ and if $U_{\alpha}\cap U_{\beta}\cap U_{\gamma} \neq \emptyset,$ the transition functions satisfy the cocycle identity 
\begin{align}\label{cocycle condition def}
g_{\alpha\beta}(x)g_{\beta\gamma}(x) = g_{\alpha\gamma}(x) \qquad \text{ for every } x \in U_{\alpha}\cap U_{\beta}\cap U_{\gamma}.
\end{align}
\paragraph*{\textbf{Bundles as transition function data}} We shall be using an equivalent way ( see e.g. \cite{Steenrod_fibrebundles} ) of defining the bundle structure -- by specifying the open cover $\mathcal{U}$ along with the cocycles $\left\lbrace g_{\alpha\beta} \right\rbrace_{\alpha, \beta \in I}.$ $P = \left( \left\lbrace U_{\alpha}\right\rbrace_{\alpha \in I}, \left\lbrace g_{\alpha\beta} \right\rbrace_{\alpha, \beta \in I} \right)$ shall denote a smooth or $C^{0}$ principal $G$ bundle, if $g_{\alpha\beta}$ are smooth or continuous, respectively. We denote the space of smooth and $C^{0}$ principal $G$-bundles over $M^{n}$ by the notation $\mathcal{P}^{0}_{G}\left(M^{n}\right)$ and $\mathcal{P}^{\infty}_{G}\left(M^{n}\right)$ respectively.  
 
\paragraph*{\textbf{Good covers and refinements}} A \emph{refinement} of a cover $\left\lbrace U_{\alpha} \right\rbrace_{\alpha \in I}$ is another cover $\left\lbrace V_{j} \right\rbrace_{j \in J}$ with a \emph{refinement map} $\phi: J \rightarrow I $ such that for every $j \in J,$ we have 
$ V_{j} \subset \subset U_{\phi(j)}.$ If $\left\lbrace U_{\alpha}\right\rbrace_{\alpha \in I}$ and $ \left\lbrace V_{\tilde{\alpha}}\right\rbrace_{\tilde{\alpha} \in \tilde{I}}$ are two covers of the same base space, then a \emph{common refinement} is another cover $\left\lbrace W_{j}\right\rbrace_{j \in J}$ with refinement maps $\phi: J  \rightarrow I$ and $\tilde{\phi}: J \rightarrow \tilde{I}$ such that for every $j \in J,$ we have 
$ W_{j} \subset \subset U_{\phi(j)}\cap V_{\tilde{\phi}(j)}.$ 

\begin{notation}
	We shall always assume the covers ( including refinements and common refinements ) involved are \emph{finite} and \emph{good cover in the \u{C}ech sense}, or simply a \emph{good cover}, i.e. every nonempty finite intersection of the open sets in the elements of the cover are diffeomorphic to the open unit Euclidean ball. In fact, we shall assume that the elements in the cover are small enough convex geodesic balls such that their volume is comparable to Euclidean balls.  
\end{notation}

\paragraph*{\textbf{Connection, gauges and curvature}}
A \emph{connection}, or more precisely, a \emph{connection form} $A$ on $P$ is a collection $\left\lbrace A_{\alpha} \right\rbrace_{\alpha \in I},$ where $A_{\alpha}: U_{\alpha} \rightarrow \Lambda^{1}\mathbb{R}^{n}\otimes \mathfrak{g}$ satisfy the \textbf{gluing relations} 
\begin{align}\label{gluing relation def}
A_{\beta} = g_{\alpha\beta}^{-1} dg_{\alpha\beta} +  g_{\alpha\beta}^{-1} A_{\alpha}g_{\alpha\beta} \qquad \text{ a.e. in } U_{\alpha}\cap U_{\beta}.
\end{align} They define a global $\mathfrak{g}$-valued $1$-form   $A: M^{n} \rightarrow \Lambda^{1}T^{\ast}M^{n}\otimes \mathfrak{g},$ which is smooth if $A_{\alpha}$s are.  We denote the space of smooth connections on a $P$ by the notation $\mathcal{A}^{\infty}\left(P\right).$ A \emph{gauge} $\rho =\left\lbrace \rho_{\alpha} \right\rbrace_{\alpha \in I}$ is a collection of maps $\rho_{\alpha}: U_{\alpha} \rightarrow G$. which represents a change of trivialization for the bundle, given by 
$$\phi_{\alpha}^{\rho_{\alpha}} (x, h) = \phi_{\alpha}(x,\rho_{\alpha}(x)h ) \qquad \text{ for all } x \in U_{\alpha} \text{ and for all } h \in G . $$ Then the new transition functions are given by $ h_{\alpha\beta} = \rho_{\alpha}^{-1}g_{\alpha\beta}\rho_{\beta}$ in $ U_{\alpha}\cap U_{\beta}$ for all $\alpha, \beta \in I.$ The local representatives of the connections form with respect to the new trivialization $\left\lbrace A_{\alpha}^{\rho_{\alpha}}\right\rbrace_{\alpha \in I}$ satisfy the gauge change identity
\begin{align}
\label{gauge change relation}
A_{\alpha}^{\rho_{\alpha}} = \rho_{\alpha}^{-1}d\rho_{\alpha} + \rho_{\alpha}^{-1} A_{\alpha}\rho_{\alpha} \qquad \text{ a.e. in } U_{\alpha}, \quad \text{ for all }\alpha \in I.
\end{align} 
 The \emph{curvature} or the \emph{curvature form} associated to a connection form $A$ is a $\mathfrak{g}$-valued $2$-form on $M^{n},$ denoted $F_{A}:M^{n}\rightarrow \Lambda^{2}T^{\ast}M^{n}\otimes \mathfrak{g}.$ Its local expressions, $\left( F_{A}\right)_{\alpha \in I},$ denoted $F_{A_{\alpha}}$ by a slight abuse of notations, are given by 
\begin{align}
\label{curvature of a conn def}
F_{A_{\alpha}} = dA_{\alpha} + A_{\alpha}\wedge A_{\alpha} = dA_{\alpha} + \frac{1}{2} \left[ A_{\alpha}, A_{\alpha}\right] \qquad \text{ in } U_{\alpha} , \quad \text{ for all }\alpha \in I,
\end{align}
where the wedge product denotes the wedge product of $\mathfrak{g}$-valued forms and the bracket $\left[ \cdot, \cdot \right]$ is the Lie bracket of $\mathfrak{g},$ extended to $\mathfrak{g}$-valued forms the usual way. The gauge change identity \eqref{gauge change relation} implies 
\begin{align}
\label{connection gauge change def}
F_{A_{\alpha}^{\rho_{\alpha}}} = \rho_{\alpha}^{-1} F_{A_{\alpha}}\rho_{\alpha} \qquad \text{ a.e. in } U_{\alpha} , \quad \text{ for all }\alpha \in I.
\end{align}
Similarly, the gluing relation \eqref{gluing relation def} implies that we have 
$F_{A_{\beta}} = g_{\alpha\beta}^{-1} F_{A_{\alpha}}g_{\alpha\beta}$ in $ U_{\alpha}\cap U_{\beta},$ whenever $U_{\alpha}\cap U_{\beta}\neq \emptyset.$ This implies  $\left( F_{A}\right)_{\alpha \in I}$ defines a global $\mathfrak{g}$-valued $2$-form on $M^{n}.$   

\paragraph*{\textbf{Yang-Mills energy}} For any $1 \leq q < \infty,$ the $q$-Yang-Mills energy of a connection $A,$ denoted $YM_{q}\left( A\right),$ is defined as 
$$ YM_{q}\left( A\right): = \int_{M^{n}} \left\lvert F_{A}\right\rvert^{q},$$
where the norm $\left\lvert \cdot \right\rvert$ denotes the norm for \emph{$\mathfrak{g}$-valued differential forms}. \eqref{connection gauge change def} implies ( see e.g. \cite{Tu_connections} ) that the norm $\left\lvert F_{A} \right\rvert$ is gauge invariant and thus the integrand in $YM_{q}$ is gauge invariant for any $1 \leq q < \infty.$  
\subsection{Sobolev bundles and connections}
 \paragraph{Sobolev principal $G$ bundles} Now we define bundles where the transition functions are Sobolev maps, not necessarily smooth or continuous. See Appendix \ref{Gvaluedmaps} for more on $G$-valued Sobolev maps. In analogy with the case of smooth bundles, we define 
\begin{definition}[$W^{k,p}$ principal $G$-bundles]
	We call $P$ a $W^{k,p}$ principal $G$-bundle over $M^{n},$ denoted by $P \in \mathcal{P}^{k,p}_{G}\left(M^{n}\right),$  if $P = \left( \left\lbrace U_{\alpha}\right\rbrace_{\alpha \in I}, \left\lbrace g_{\alpha\beta} \right\rbrace_{\alpha, \beta \in I} \right),$ where $g_{\alpha\beta} \in W^{k,p} \left( U_{\alpha}\cap U_{\beta}; G\right) $ for every $\alpha, \beta \in I$ with $U_{\alpha}\cap U_{\beta} \neq \emptyset, $ and for every $\alpha, \beta, \gamma \in I$ such that $U_{\alpha}\cap U_{\beta}\cap U_{\gamma} \neq \emptyset, $ the transition maps satisfy $g_{\alpha\alpha}= \mathbf{1}_{G}$ for every $\alpha \in I$ and the cocycle conditions 
	\begin{align}
	\label{cocycle sobolev def}
	g_{\alpha\beta}(x)g_{\beta\gamma}(x) = g_{\alpha\gamma}(x) \qquad \text{ for a.e. } x \in U_{\alpha}\cap U_{\beta}\cap U_{\gamma}.
	\end{align} \end{definition}
We also need the notion of Sobolev equivalence of Sobolev bundles, which is just another name for being gauge related by Sobolev gauge changes.
\begin{definition}[$W^{k,p}$ equivalence]
	Two $W^{k,p}$ principal $G$-bundles $P$ and $\tilde{P}$ over the same base space $M^{n}$ are \emph{$W^{k,p}$ equivalent}, denoted by $P \stackrel{W^{k,p}}{\simeq} \tilde{P},$ if there exists a common refinement $\left\lbrace W_{j}\right\rbrace_{j \in J}$ of the covers $\left\lbrace U_{\alpha}\right\rbrace_{\alpha \in I}$ and $\left\lbrace V_{\tilde{\alpha}}\right\rbrace_{\tilde{\alpha} \in \tilde{I}} $ and maps 
	$\sigma_{j} \in W^{k,p}\left( W_{j}; G\right)$ for each $j \in J$  such that 
	\begin{align}\label{equivalence gauge relation}
	h_{\tilde{\phi}(i)\tilde{\phi}(j)} = \sigma_{i}^{-1}g_{\phi(i)\phi(j)}\sigma_{j}\qquad \text{ a.e. in } W_{i}\cap W_{j},
	\end{align} 
	for each pair $i,j \in J$ with $W_{i}\cap W_{j} \neq \emptyset,$ where $\phi: J  \rightarrow I$ and $\tilde{\phi}: J \rightarrow \tilde{I}$ are the respective refinement maps and $\left\lbrace g_{\alpha\beta} \right\rbrace_{\alpha, \beta \in I} $ and $\left\lbrace h_{\tilde{\alpha}\tilde{\beta}} \right\rbrace_{\tilde{\alpha}, \tilde{\beta} \in \tilde{I}}$ are the respective transition maps. We shall write $P \stackrel{W^{k,p}}{\simeq}_{\sigma} \tilde{P}$ to specify the equivalence map.  
\end{definition} 
Smooth or $C^{0}$ equivalence is defined in analogous manner by requiring the maps $\sigma_{j}$ to be smooth or $C^{0}$ respectively. It is easy to check that they are indeed equivalence relations in the corresponding category. If $P$ is a $C^{0}$-bundle, we denote its equivalence class under $C^{0}$-equivalence by $\left[P\right]_{C^{0}}.$   
\paragraph{Sobolev spaces of connections}
Now we define the Sobolev space $\mathcal{U}^{k,p}$ of connection form on a $W^{k,p}$ bundle $P \in \mathcal{P}^{k,p}_{G}\left(M^{n}\right).$ 
\begin{definition}[The $\mathcal{U}^{k,p}$ spaces of connections]
	We say the connection $A=\left\lbrace A_{\alpha} \right\rbrace_{\alpha \in I}$ is a $\mathcal{U}^{1,p}$-connection on $P$ if we have 
	\begin{align*}
	A_{\alpha} \in L^{p}\left( U_{\alpha}; \Lambda^{1}\mathbb{R}^{n}\otimes\mathfrak{g}\right) \text{ and } dA_{\alpha} \in L^{\frac{p}{2}}\left( U_{\alpha}; \Lambda^{2}\mathbb{R}^{n}\otimes\mathfrak{g}\right) \quad \text{ for every } \alpha \in I. 
	\end{align*}
	$\mathcal{U}^{1,p}\left(P\right),$ the space of $\mathcal{U}^{1,p}$-connection on $P$, is equipped with the norm  
	$$ \left\lVert A_{\alpha}\right\rVert_{\mathcal{U}^{k,p}\left( U_{\alpha}; \Lambda^{1}\mathbb{R}^{n}\otimes\mathfrak{g}\right)}:= \left\lVert A_{\alpha}\right\rVert_{L^{p}\left( U_{\alpha}; \Lambda^{1}\mathbb{R}^{n}\otimes\mathfrak{g}\right)} + \left\lVert dA_{\alpha}\right\rVert_{L^{\frac{p}{2}}\left( U_{\alpha}; \Lambda^{1}\mathbb{R}^{n}\otimes\mathfrak{g}\right)}.$$
	We say the connection $A=\left\lbrace A_{\alpha} \right\rbrace_{\alpha \in I}$ is a $\mathcal{U}^{2,p}$-connection on $P$ if we have
	\begin{align*}
	A_{\alpha} \in W^{1,\frac{p}{2}}\left( U_{\alpha}; \Lambda^{1}\mathbb{R}^{n}\otimes\mathfrak{g}\right)  \text{ for every } \alpha \in I.
	\end{align*}
	The $\mathcal{U}^{2,p}$ norm of $A_{\alpha}$ is simply its $W^{1, \frac{p}{2}}$ norm. \end{definition}

\section{Strong Density in the critical dimension}\label{elliptic estimates density}
\subsection{Smooth Approximation in subcritical regime}
We begin with the smooth approximation theorem for Sobolev bundles in the subcritical regime $kp > n.$ The validity of the result is well known to experts, but a complete proof is difficult to find in the literature. For $kp > n,$ $W^{k,p}$ bundles are $C^{0}$ bundles. Approximating $C^{0}$ bundles by $C^{\infty}$ ones are classical and one can, in particular, use the heavy machinery of classifying spaces ( cf. \cite{Husemoller_fibrebundles} ). But since we need to keep control of the Sobolev norms, it is unclear whether such an approach can be used in the Sobolev setting. On the other hand, one can smooth continuous bundles `by hand' ( see e.g. \cite{MullerWockel_equivalencebundles}, also \cite{NaberG_gaugetheorybook}, \cite{Steenrod_fibrebundles} ) and this approach is more amenable to the modifications needed to work in the Sobolev setting. Our proof here follows this road and adapt the arguments in \cite{MullerWockel_equivalencebundles} ( for the infinite dimensional case ) to work in our finite dimensional but Sobolev setting. As far as we are aware, this proof is new. But since this somewhat digresses from the main goal of our article, it is relegated to the Appendix \ref{proof of subcritical approx}.      
\begin{theorem}[Smooth approximation in subcritical regime]\label{smoothingC0bundle} 
	Given any $P \in \mathcal{P}_{G}^{k,p}\left(M^{n}\right)$ with $kp > n$ and any $\varepsilon > 0,$ there is a smooth principal $G$-bundle $P^{\varepsilon} \in \mathcal{P}_{G}^{\infty}\left(M^{n}\right)$ such that $P^{\varepsilon}$ is $\varepsilon$-close to $P$ in $W^{k,p}$ norm and  $P^{\varepsilon} \stackrel{W^{k,p}}{\simeq} P $ hold. Moreover, the $C^{0}$-equivalence maps can be chosen to lie in the $\varepsilon$-neighborhood of the identity element of $G$ in $C^{0}$ norm on each bundle chart. \smallskip 
	
	\noindent More precisely, if $P = \left( \left\lbrace U_{\alpha}\right\rbrace_{\alpha \in I}, \left\lbrace g_{\alpha\beta} \right\rbrace_{\alpha, \beta \in I} \right) $, then there exists a good refinement $\left\lbrace V_{j}\right\rbrace_{j \in J}$ of $\left\lbrace U_{\alpha}\right\rbrace_{\alpha \in I}$ such that there exists \textbf{continuous} maps $\sigma_{j} \in W^{k,p}\left(V_{j}; G\right)$ and 
	\textbf{smooth} transition maps $h_{ij}\in C^{\infty}\left(V_{i}\cap V_{j}; G\right)  $ for all $i,j \in J,$ whenever the intersection is non-empty, satisfying  
	\begin{itemize}
		\item[(i)] $h_{ij}h_{jk} = h_{ik} \quad \text{ for a.e. } x \in V_{ijk} \text{ whenever } V_{ijk} \neq \emptyset, $
		\item[(ii)] $h_{ij} = \sigma_{i}^{-1} g_{\phi(i)\phi(j)} \sigma_{j}  \qquad \text{ for a.e. } x \in V_{ij} \text{ whenever } V_{ij} \neq \emptyset.$ 
		\item[(iii)] $ \left\lVert h_{ij} - g_{\phi(i)\phi(j)}\right\rVert_{W^{k,p}\left( V_{ij};G\right)}  \leq \varepsilon$ whenever $V_{ijk} \neq \emptyset,$ where $\phi: J \rightarrow I$ is the refinement map.
		\item[(iv)] $ \left\lVert \sigma_{j} - \mathbf{1}_{G}\right\rVert_{L^{\infty}\left( V_{j}; G \right)} , \left\lVert d\sigma_{j} \right\rVert_{W^{k-1,p}\left( V_{j}; \Lambda^{1}\mathbb{R}^{n}\otimes\mathfrak{g} \right)} \leq \varepsilon $  for every $ j \in J.$ 
	\end{itemize}
\end{theorem}
\begin{remark}
	Conclusion (i) simply expresses the fact that we indeed have a well-defined bundle structure on $P^{\varepsilon}.$ Conclusion (ii) encodes the assertion that $P^{\varepsilon}$ is equivalent to $P$. Conclusion (iii) is the precise meaning of $P^{\varepsilon}$ being $\varepsilon$-close to $P$ in $W^{k,p}$ norm. The estimate in conclusion (iv) is essentially equivalent to (iii), as was already essentially proved by Uhlenbeck in \cite{UhlenbeckGaugefixing}, Corollary 3.3.     
\end{remark}
\subsection{Coulomb gauges and elliptic estimates}
\begin{notation}
	$\mathfrak{M}\left( N\right)$ denote the space of $N\times N$ matrices and for $U \subset \mathbb{R}^{n}$ open and bounded, the notation $L^{\left(s, \theta \right)}\left( U; \mathbb{R}^{N}\right)$ for any $1 < s < \infty$ and any $ 1 \leq \theta < \infty$ will denote the Lorentz space of maps $\left\lbrace f: U \rightarrow \mathbb{R}^{N}:\left\lVert f \right\rVert_{L^{\left( s,\theta \right)}\left(U; \mathbb{R}^{N} \right)} < \infty \right\rbrace,$ which is a Banach space ( see \cite{SteinWeiss_Fourieranalysis} ) with a norm equivalent to the quasinorm 
	$$ \left\lVert f \right\rVert_{L^{\left( s,\theta \right)}\left(U; \mathbb{R}^{N} \right)}^{\theta} 
	:= \int_{0}^{\infty} \left[ t^{s} \operatorname{meas}\left( \left\lbrace x \in U : \left\lvert f \right\rvert > t \right\rbrace\right) \right]^{\frac{\theta}{s}} \frac{\mathrm{d}t}{t}. $$
\end{notation}
\noindent Now we start with the elliptic estimates. The following lemmas are crucial for what we shall be doing in the rest of the article. They will be used to prove regularity of bundles in which the connection is in the Coulomb gauge in the critical dimension. Continuity of Coulomb bundles is first observed by Taubes \cite{Taubes_YangMillsModulispaces} for $n=4$ and for any $n \geq 4$ by Rivi\`{e}re \cite{Riviere_QuantizationYangMills}. Here we show such bundles are $C^{0, \alpha}$ bundle for any $\alpha < 1.$ However, $L^{p}$ version of Lemma \ref{ellipticCritical} has been used earlier in this context by Shevchishin in \cite{Shevchishin_limitholonomyYangMills} to prove H\"{o}lder continuity of the Coulomb bundles. But these results do not seem to be widely known.   
\begin{lemma}\label{elliptic estimate with zero boundary value}
Let $\Omega \subset \mathbb{R}^{n}$ be a bounded, open and smooth subset and suppose $ A \in L^{n}\left(\Omega;\Lambda^{1}\mathbb{R}^{n}\otimes \mathfrak{M}\left( N\right)\right).$ If $\alpha \in W_{0}^{1,2}\left( \Omega; \mathbb{R}^{N}\right)$ satisfies \begin{align}\label{critical elliptic with zero boundary}
\Delta \alpha = A\cdot \nabla \alpha + F \qquad \text{ in } \Omega,
\end{align}
with $F \in L^{q}\left( \Omega; \mathbb{R}^{N}\right)$ for some $ \frac{2n}{n+2} \leq q < n, $ or respectively, $F \in L^{\left(s, \theta \right)}\left( \Omega; \mathbb{R}^{N}\right)$ for some  $ \frac{2n}{n+2} < s < n $ and $1 \leq \theta < \infty ,$ then there exists a small constant $\varepsilon_{1}= \varepsilon_{1}\left( n,N,q, \Omega \right) > 0,$ respectively, $\varepsilon_{1}= \varepsilon_{1}\left( n,N,s, \theta, \Omega \right) > 0,$ such that if $$ \left\lVert A \right\rVert_{L^{n}\left(\Omega;\Lambda^{1}\mathbb{R}^{n}\otimes \mathfrak{M}\left( N\right)\right)} \leq \varepsilon_{1},$$ then $\alpha \in W^{2,q}\left( \Omega; \mathbb{R}^{N}\right),$ respectively $\alpha \in W^{2,\left(s, \theta \right)}\left( \Omega; \mathbb{R}^{N}\right),$ and there exists a constant $C_{\Omega}= C_{\Omega}\left( n,N,q, \Omega \right) \geq 1, $ respectively $C_{\Omega}= C_{\Omega}\left( n,N,s,\theta, \Omega \right) \geq 1, $ such that we have the estimate 
\begin{align}
\left\lVert \alpha \right\rVert_{W^{2,q}\left( \Omega; \mathbb{R}^{N}\right)} &\leq C_{\Omega} \left\lVert F \right\rVert_{L^{q}\left( \Omega; \mathbb{R}^{N}\right)}, \\ \intertext{respectively,} 
\left\lVert \alpha \right\rVert_{W^{2,\left(s, \theta \right)}\left( \Omega; \mathbb{R}^{N}\right)} &\leq C_{\Omega} \left\lVert F \right\rVert_{L^{\left(s, \theta \right)}\left( \Omega; \mathbb{R}^{N}\right)}. 
\end{align}
Furthermore, the smallness parameter $\varepsilon_{1}$ is scale invariant. More precisely, if $\Omega_{r}= \left\lbrace rx : x \in \Omega \right\rbrace $ is a rescaling of $\Omega,$ then $\Omega$ and $\Omega_{r}$ has the same $\varepsilon_{\Delta_{Cr}}.$ 
\end{lemma}
\begin{proof}
The proof is a fixed point argument coupled with uniqueness. We only prove the Lorentz case.  With $s$ and $\theta$ as in the lemma, for any $v \in W^{2,\left(s, \theta \right)}\cap W_{0}^{1,2}$, let $T(v) \in W_{0}^{1,2}$ be the solution of the equation $\Delta \left( T(v) \right) = A.\nabla v + F.$  Since by Peetre-Sobolev embedding ( see \cite{Peetre_Lorentzembedding}, \cite{Tartar_Lorentzembeddings} ), $W^{2,\left(s, \theta \right)} \hookrightarrow W^{1, \left(\frac{ns}{n-s}, \theta \right)}$ and $L^{n}= L^{\left(n, n \right)},$ by H\"{o}lder inequality for Lorentz spaces the term $A.\nabla v $ in the right hand side is $L^{\left(s, \frac{n\theta}{n+\theta} \right)} \hookrightarrow L^{\left(s, \theta \right)} $ if $\theta > \frac{n}{n-1}$ or $L^{\left(s, 1 \right)} \hookrightarrow L^{\left(s, \theta \right)}$ otherwise. Since $f \in L^{\left(s, \theta \right)},$ by the usual $L^{p}$ estimate for the Laplacian, which extends by interpolation to Lorentz spaces ( see \cite{SteinWeiss_Fourieranalysis} ), we conclude $\nabla^{2}T(v) \in L^{\left(s, \theta \right)}$. Noting that the $L^{\left(s, \theta \right)}$ norm of the Hessian is an equivalent norm on $W^{2,\left(s, \theta \right)}\cap W_{0}^{1,2}$, we deduce $T(v) \in W^{2,\left(s, \theta \right)}$ along with the estimate $$ \left\lVert T(v)- T(w)\right\rVert_{W^{2,\left(s, \theta \right)}} \leq C_{q} \left\lVert A \right\rVert_{L^{n}}\left\lVert v- w\right\rVert_{W^{2,\left(s, \theta \right)}} $$ for any $v,w \in W^{2,\left(s, \theta \right)}.$ Then we can choose $\left\lVert A \right\rVert_{L^{n}}$ small enough such that $T$ is a contraction and conclude the existence of an unique fixed point $v_{0} \in W^{2,\left(s, \theta \right)}$ by Banach fixed point theorem. Since both $v_{0}$ and $\alpha$ are $W_{0}^{1,2}$ solutions of \eqref{critical elliptic with zero boundary} and we have the estimate  
\begin{align*}
\left\lVert \alpha - v_{0}\right\rVert_{W^{1,2}} \leq C \left\lVert \alpha - v_{0} \right\rVert_{W^{2,\frac{2n}{n+2}}}\leq C \left\lVert A \right\rVert_{L^{n}}\left\lVert \alpha - v_{0}\right\rVert_{W^{1,2}}, 
\end{align*} we can choose $\left\lVert A \right\rVert_{L^{n}}$ small enough such that $C \left\lVert A \right\rVert_{L^{n}} < 1,$ which forces $\alpha=v_{0}.$ Thus $\alpha \in W^{2,\left(s, \theta \right)}\left( \Omega; \mathbb{R}^{N}\right)$ and we get the estimate 
$$ \left\lVert \alpha \right\rVert_{W^{2,\left(s, \theta \right)}} \leq C_{s,\theta} \left( \left\lVert A \right\rVert_{L^{n}}\left\lVert \alpha \right\rVert_{W^{2,\left(s, \theta \right)}} + \left\lVert F \right\rVert_{L^{\left(s, \theta \right)}}\right). $$ But since  $C_{s,\theta} \left\lVert A \right\rVert_{L^{n}} < 1, $ setting $C_{\Omega} = \frac{C_{s,\theta}}{\left( 1 - C_{s,\theta} \left\lVert A \right\rVert_{L^{n}}\right)}$ proves the lemma except the claim about scaling. Now if $\alpha, A, F$ satisfies \eqref{critical elliptic with zero boundary} in $\Omega_{r},$ then the rescaled maps $\tilde{\alpha}(x):= \alpha (rx), \tilde{A}(x):= r A(rx),$ $\tilde{F}:= r^{2}F(rx)$ satisfies \eqref{critical elliptic with zero boundary} in $\Omega.$ The scale invariance follows from the equality of $L^{n}$ norms of $\tilde{A}$ and $A.$
\end{proof}
\begin{lemma}[Elliptic estimate in critical setting]\label{ellipticCritical}
	Let $\Omega \subset \mathbb{R}^{n}$ be a bounded, open set and $ A \in L^{n}\left(\Omega;\Lambda^{1}\mathbb{R}^{n}\otimes \mathfrak{M}\left( N\right)\right).$ Let $f \in L^{p}\left( \Omega; \mathbb{R}^{N}\right)$ for some $ \frac{2n}{n+2} \leq p < n, $ or respectively, $f \in L^{\left(q, \theta \right)}\left( \Omega; \mathbb{R}^{N}\right)$ for some  $ \frac{2n}{n+2} < q < n $ and $1 \leq \theta < \infty .$ Then there exists a small constant $\varepsilon_{\Delta_{Cr}} = \varepsilon_{\Delta_{Cr}}\left( n,N,p, \Omega \right) > 0,$ respectively $\varepsilon_{\Delta_{Cr}} = \varepsilon_{\Delta_{Cr}}\left( n,N,q,\theta, \Omega \right) > 0,$ such that if $$ \left\lVert A \right\rVert_{L^{n}\left(\Omega;\Lambda^{1}\mathbb{R}^{n}\otimes \mathfrak{M}\left( N\right)\right)} \leq \varepsilon_{\Delta_{Cr}},$$ then for any solution $u \in W^{1,2}\left(\Omega; \mathbb{R}^{N}\right)$ of 
	\begin{equation}\label{critical elliptic eqn}
	\Delta u = A\cdot \nabla u + f \qquad \text{ in } \Omega, 
	\end{equation} we have $u \in W_{loc}^{2,p}\left(\Omega; \mathbb{R}^{N}\right),$ respectively $u \in W_{loc}^{2,\left(q, \theta \right)}\left(\Omega; \mathbb{R}^{N}\right).$ Furthermore, for any compact set $K \subset \subset \Omega, $  the exists a constant $C = C(n, N, p, \Omega, K ) \geq 1, $ respectively $C = C(n, N, q, \theta, \Omega, K ) \geq 1, $ such that we have the estimate 
	\begin{align}
	\left\lVert u \right\rVert_{W^{2,p}\left( K;\mathbb{R}^{N} \right)} &\leq C \left( \left\lVert u \right\rVert_{W^{1, 2}\left(\Omega; \mathbb{R}^{N}\right)} + \left\lVert f \right\rVert_{L^{p}\left(\Omega; \mathbb{R}^{N}\right)} \right), \\
	\intertext{respectively,}
	\left\lVert u \right\rVert_{W^{2,\left(q, \theta \right)}\left( K;\mathbb{R}^{N} \right)} &\leq C \left( \left\lVert u \right\rVert_{W^{1, 2}\left(\Omega; \mathbb{R}^{N}\right)} + \left\lVert f \right\rVert_{L^{\left(q, \theta \right)}\left(\Omega; \mathbb{R}^{N}\right)} \right).
	 \end{align}
	Moreover, the smallness parameter $\varepsilon_{\Delta_{Cr}}$ is scale invariant. More precisely, if $\Omega_{r}= \left\lbrace rx : x \in \Omega \right\rbrace $ is a rescaling of $\Omega,$ then $\Omega$ and $\Omega_{r}$ has the same $\varepsilon_{\Delta_{Cr}}.$ 
\end{lemma}
\begin{remark}
	Later on, to satisfy the smallness condition on the $L^{n}$ norm of $A,$ we are going to shrink the balls. So the scale invariance conclusion is crucial. 
\end{remark}
\begin{proof}
We localize the problem and bootstrap. We prove only the Lorentz case. To this end, we chose $m=1$ if $q \leq \frac{2n}{n-2}$ and otherwise we chose the smallest integer $m \geq 2$ such that $\frac{n(q-2)}{2q} \leq m < \frac{n}{2} $ if $\theta \geq  q$ or $\frac{n(q-2)}{2q} < m < \frac{n}{2} $ if $1 \leq \theta < q.$ For any $K \subset \subset \Omega,$ we choose open sets $\left\lbrace \Omega_{l}\right\rbrace_{1 \leq l \leq m}$ such that 
\begin{align*}
K \subset \subset \Omega_{m} \subset \subset \ldots \Omega_{l+1} \subset \subset \Omega_{l} \subset \subset \ldots \Omega_{1}\subset \subset \Omega.  
\end{align*}
Now we shall show that we can choose $\left\lVert A \right\rVert_{L^{n}}$ small enough such that for any solution $u \in W^{1,2}\left(\Omega; \mathbb{R}^{N}\right)$ of 
\eqref{critical elliptic eqn}, we have $u \in W^{1, \frac{2n}{n-2m}}\left(\Omega_{m}; \mathbb{R}^{N} \right).$ We show only the case $m \geq 2,$ the other case being easier. We prove this by induction over $l.$ We assume that $u \in W^{1, \frac{2n}{n-2l}}\left(\Omega_{l}; \mathbb{R}^{N} \right)$ for some $1 \leq l \leq m-1$ and prove that we can choose $\left\lVert A \right\rVert_{L^{n}}$ small enough such that $u \in W^{1, \frac{2n}{n-2(l+1)}}\left(\Omega_{l+1}; \mathbb{R}^{N} \right). $ Let $\phi \in C_{c}^{\infty}\left(\Omega_{l}\right)$ be a smooth cut-off function such that $\phi \equiv 1$ in a neighborhood of $\Omega_{l+1}.$ Since $u$ is a solution to \eqref{critical elliptic eqn}, $\phi u \in W_{0}^{1,2} \left(\Omega_{l}; \mathbb{R}^{N} \right)$ solves 
\begin{align}\label{localized equation}
\Delta \left(\phi u \right) = A \cdot \nabla \left(\phi u \right) + \phi f + u \left( \Delta \phi -  A \cdot \nabla \phi \right) + 2 \nabla \phi \cdot \nabla u  \qquad \text{ in } \Omega_{l}.  
\end{align}
We set $\alpha = \phi u$ and $F = \phi f + u \left( \Delta \phi -  A \cdot \nabla \phi \right) + 2 \nabla \phi \cdot \nabla u $ and plan to use lemma \ref{elliptic estimate with zero boundary value}. Note that by our choice of $m,$ the integrability of $F$ is determined by the least regular terms, which are in $L^{\frac{2n}{n-2l}}\left(\Omega_{l}; \mathbb{R}^{N} \right).$ Thus, using lemma \ref{elliptic estimate with zero boundary value} with $q = \frac{2n}{n-2l},$ we obtain $\phi u \in W^{2, \frac{2n}{n-2l}}\left(\Omega_{l}; \mathbb{R}^{N} \right)$ and thus by Sobolev embedding, $ \phi u  \in W^{1, \frac{2n}{n-2(l+1)}}\left(\Omega_{l}; \mathbb{R}^{N} \right).$ Since $\phi$ is identically $1$ in a neighborhood of $\Omega_{l+1},$ this proves the induction step. The same argument shows that $u \in W^{1, \frac{2n}{n-2}}\left(\Omega_{1}; \mathbb{R}^{N} \right),$ since $u \in W^{1,2}\left(\Omega; \mathbb{R}^{N} \right)$ and thus we can start the induction. This establishes that $u \in W^{1, \frac{2n}{n-2m}}\left(\Omega_{m}; \mathbb{R}^{N} \right)$ if $\left\lVert A \right\rVert_{L^{n}}$ is small enough. 

Once again we choose a smooth cut-off function $\phi \in C_{c}^{\infty}\left(\Omega_{m}\right)$ such that $\phi \equiv 1$ in a neighborhood of $K.$ Once again, $\phi u $ satisfies \eqref{localized equation} in $\Omega_{m}.$ But since $m > \frac{n(q-2)}{2q}$ if $1 \leq \theta < q,$ this time the integrability of $F$ is determined by the first term $\phi f,$ which is $L^{\left(q, \theta \right)}\left(\Omega_{m}; \mathbb{R}^{N} \right).$ Thus applying Lemma \ref{elliptic estimate with zero boundary value} once again, we deduce that $\phi u \in W^{2,\left(q, \theta \right)}\left(\Omega_{m}; \mathbb{R}^{N} \right)$ and thus $u \in W^{2,\left(q, \theta \right)}\left(K; \mathbb{R}^{N} \right)$ if $\left\lVert A \right\rVert_{L^{n}}$ is small enough. We finally choose the smallness parameter to be the minimum of the smallness parameters in the finitely many steps. Combining the estimates in each step yields 
$$\left\lVert u \right\rVert_{W^{2,\left(q, \theta \right)}\left( K;\mathbb{R}^{N} \right)} \leq C \left(  \left\lVert u \right\rVert_{W^{1,2}\left( \Omega;\mathbb{R}^{N} \right)}+ \left\lVert f \right\rVert_{L^{\left(q, \theta \right)}\left(\Omega; \mathbb{R}^{N}\right)} \right)  . $$
This concludes the proof. The scale invariance can be shown as before. \end{proof} 
\begin{lemma}[Coulomb gauges]\label{Coulombgauge}
	Let $r>0$ be a real number, $x_{0} \in \mathbb{R}^{n}$ and let $B_{r}(x_{0})\subset \mathbb{R}^{n}$ be the ball of radius $r$ around $x_{0}$. Then there exist constants $\varepsilon_{Coulomb} = \varepsilon_{coulomb}\left( G, n \right)  > 0 $ and $C_{Coulomb}= C_{Coulomb}\left( G, n \right) \geq 1 $ such that for any $A \in \mathcal{U}^{1,n}\left(B_{r}(x_{0}) \right)$ with $$ \left\lVert F_{A}\right\rVert_{L^{\frac{n}{2}}\left( B_{r}(0);\Lambda^{2} T^{\ast}B_{r}(0)\otimes \mathfrak{g} \right)} \leq \varepsilon_{Coulomb},$$
	there exists $\rho \in W^{1,n}\left(B_{r}(x_{0}) ; G \right)$ such that 
	\begin{align*}
	\left\lbrace \begin{aligned}
	&d^{\ast} A^{\rho} = 0 \quad \text{ in } B_{r}(x_{0}),  \\
&\iota^{\ast}_{\partial B_{r}(x_{0}) } \left( \ast A^{\rho}\right)  = 0 \quad \text{ on } \partial B_{r}(x_{0})	\end{aligned}\right.
	\end{align*}
	and we have the estimates 
\begin{multline*}
\left\lVert \nabla A^{\rho}\right\rVert_{L^{\frac{n}{2}}\left( B_{r}(x_{0});\mathbb{R}^{n\times n}\otimes \mathfrak{g} \right)} + \left\lVert  A^{\rho}\right\rVert_{L^{n}\left( B_{r}(x_{0});\Lambda^{1}\mathbb{R}^{n}\otimes \mathfrak{g} \right)} \\ \leq C_{Coulomb}\left\lVert F_{A}\right\rVert_{L^{\frac{n}{2}}\left( B_{r}(x_{0});\Lambda^{2}\mathbb{R}^{n}\otimes \mathfrak{g} \right)} 
\end{multline*}	
and 
\begin{multline*}
 \left\lVert  d\rho\right\rVert_{L^{n}\left( B_{r}(x_{0});\Lambda^{1}\mathbb{R}^{n}\otimes \mathfrak{g} \right)} \\ \leq C_{G}\left( C_{Coulomb}\left\lVert F_{A}\right\rVert_{L^{\frac{n}{2}}\left( B_{r}(x_{0});\Lambda^{2}\mathbb{R}^{n}\otimes \mathfrak{g} \right)}  + \left\lVert  A\right\rVert_{L^{n}\left( B_{r}(x_{0});\Lambda^{1}\mathbb{R}^{n}\otimes \mathfrak{g} \right)} \right),
\end{multline*}	
where $C_{G} \geq 1 $ is an $L^{\infty}$ bound for $G.$
\end{lemma}
\begin{proof}
	The technique of the proof is by now completely standard and goes back to Uhlenbeck \cite{UhlenbeckGaugefixing}. We stated the theorem for a ball of radius $r$ to emphasize the scale invariance. Indeed, if $A: B_{r}(0) \rightarrow \Lambda^{1}\mathbb{R}^{n}\otimes \mathfrak{g}$ is a connection, the rescaled connection $\tilde{A}(x):= r A(rx)$ is a connection on $B_{1}(0)$ with curvature $F_{\tilde{A}}(x) = r^{2} F_{A}(rx)$ and one can easily check the identities
	\begin{align*}
	\left\lVert \tilde{A} \right\rVert_{L^{n}\left(B_{1}(0)\right)} = \left\lVert A \right\rVert_{L^{n}\left(B_{r}(0)\right)} \quad \text{ and } \quad \left\lVert F_{\tilde{A}} \right\rVert_{L^{\frac{n}{2}}\left(B_{1}(0)\right)} = \left\lVert F_{A} \right\rVert_{L^{\frac{n}{2}}\left(B_{r}(0)\right)} .
	\end{align*} The translation invariance is of course obvious. The last estimate for the gauges are usually not stated explicitly, but follows rather easily from the identity
	$ d\rho = \rho A^{\rho} - A \rho .$ 
\end{proof}
\begin{remark}
	We stated the result for $\mathcal{U}^{1,n}$ connections and for Euclidean balls. If instead $A \in \mathcal{U}^{2,\frac{n}{2}},$ then similar arguments show that under the hypotheses of the theorem, there exists $\rho \in W^{k,p}\left(B_{r}(x_{0}) ; G \right)$ such that $d^{\ast} A^{\rho} = 0 $ in $B_{r}(x_{0})$, $\iota^{\ast}_{\partial B_{r}(x_{0}) } \left( \ast A^{\rho}\right)  = 0 $ on $ \partial B_{r}(x_{0})$ and $A^{\rho} \in \mathcal{U}^{2,\frac{n}{2}}$ and we have the estimates 
	\begin{equation*}
	\left\lVert \nabla A^{\rho}\right\rVert_{L^{\frac{n}{2}}} + \left\lVert  A^{\rho}\right\rVert_{L^{n}}  \leq C_{Coulomb}\left\lVert F_{A}\right\rVert_{L^{\frac{n}{2}}} 
	\end{equation*}	
	and 
	\begin{equation*}
	\left\lVert  d\rho\right\rVert_{W^{1,\frac{n}{2}}} \\ \leq C_{G}\left( C_{Coulomb}\left\lVert F_{A}\right\rVert_{L^{\frac{n}{2}}}  + \left\lVert  A\right\rVert_{W^{1,\frac{n}{2}}} \right).
	\end{equation*}	Both results extend to small geodesic balls on closed Riemannian manifolds. 
\end{remark}
\subsection{Regularity of Coulomb bundles}
Now we prove that in the critical dimension, a bundle in which a Sobolev connection is in the Coulomb gauge is actually a H\"{o}lder continuous bundle. 
\begin{theorem}[H\"{o}lder continuity of Coulomb bundles]\label{regularity Coulomb}
	Let $kp=n.$ Let $P$ be a $W^{k,p}$ principal $G$-bundle and $A \in \mathcal{U}^{k,p}\left( P \right)$ be connection on $P$ which is Coulomb, then $P$ is a $W^{2,q}\cap C^{0,\alpha}$-bundle for any $\frac{n}{2} < q < n$ and $\alpha < 1.$ More precisely, if $\left( P, A \right) = \left( \left\lbrace U_{\alpha}\right\rbrace_{\alpha \in I}, \left\lbrace g_{\alpha\beta} \right\rbrace_{\alpha, \beta \in I}, \left\lbrace A_{\alpha}\right\rbrace_{\alpha \in I} \right) $ such that $d^{\ast}A_{\alpha} = 0 $ in $ U_{\alpha}$ for every $\alpha \in I,$ then there exists a good refinement $\left\lbrace V_{j}\right\rbrace_{j \in J}$ of $\left\lbrace U_{\alpha}\right\rbrace_{\alpha \in I}$ with $V_{j} \subset \subset U_{\phi (j)}$ for every $j \in J,$ where $\phi: J \rightarrow I$ is the refinement map,  such that we have  $g_{\phi(i)\phi(j)} \in  W^{2,q}\left(V_{i}\cap V_{j}; G\right)$ for any $ \frac{n}{2} < q < n$ and thus also  $C^{0,\alpha}\left(\overline{V_{i}\cap V_{j}}; G\right)$ for any $\alpha < 1,$ for all $i,j \in J,$ whenever $V_{i}\cap V_{j} \neq \emptyset$.
\end{theorem} 
\begin{proof}
	We choose a good refinement $\left\lbrace V_{j}\right\rbrace_{j \in J}$ of $\left\lbrace U_{\alpha}\right\rbrace_{\alpha \in I}$ in such a way that there is an enlarged cover $\left\lbrace V^{'}_{j}\right\rbrace_{j \in J}$ which is also a refinement of $\left\lbrace U_{\alpha}\right\rbrace_{\alpha \in I}$ with the same refinement map $\phi: J \rightarrow I$ and we have 
	\begin{align}
	\left\lVert A_{j}\right\rVert_{L^{n}\left( V^{'}_{j};\Lambda^{1} \mathbb{R}^{n}\otimes \mathfrak{g} \right)} &<\frac{\varepsilon_{\Delta_{Cr}}}{4}  \label{LnsmallnessCoulomb}
	\end{align} for every $ j \in J,$ where $A_{j}:= A_{\phi(j)}|_{V^{'}_{j}},$ and 
	 we have $$ V_{j} \subset\subset V^{'}_{j} \subset U_{\phi(j)} \quad \text{ for every } j \in J, \qquad \bigcup\limits_{j \in J} V_{j} = \bigcup\limits_{\alpha \in I} U_{\alpha} = M^{n} .$$ 
	 Now, setting $h_{ij}= g_{\phi(i)\phi(j)}$ for every $i,j \in J$ such that $V^{'}_{i}\cap V^{'}_{j} \neq \emptyset,$ we have the gluing relations,  
	 \begin{equation}\label{gluing identityCoulomb}
	 A_{j} = h_{ij}^{-1}dh_{ij} + h_{ij}^{-1}A_{i}h_{ij} \qquad \text{ for a.e } x \text{ in } V^{'}_{i}\cap V^{'}_{j} \text{ whenever }V^{'}_{i}\cap V^{'}_{j} \neq \emptyset.
	 \end{equation}
	 Rewriting \eqref{gluing identityCoulomb}, we have, 
	 $$ dh_{ij} = h_{ij}A_{j} -  A_{i} h_{ij} \qquad \text{ for a.e } x \text{ in } V^{'}_{ij} \text{ whenever }V^{'}_{ij} \neq \emptyset .$$
	 Also, since $A$ is in Coulomb gauge, we have  $$d^{\ast}A_{i} = 0 = d^{\ast}A_{j}  \qquad \text{ in } V^{'}_{ij} .$$
	 This implies, 
	 \begin{equation}
	 - \Delta h_{ij} = \ast \left[ dh_{ij} \wedge \left( \ast A_{j}\right)\right] + \ast \left[  \left( \ast A_{i}\right) \wedge dh_{ij} \right]   \qquad \text{ in } V^{'}_{ij} .
	 \end{equation}
	 This is of the same form as \eqref{critical elliptic eqn} with $f=0.$
	 Now, we have, $$ \left\lVert A_{l} \right\rVert_{L^{n}} \leq \frac{1}{4}\varepsilon_{\Delta_{Cr}} \qquad \text{ for any } l \in J. $$
	 Thus, we can apply lemma \ref{ellipticCritical} with $f =0$ and deduce that $h_{ij} \in W^{2,p}_{loc}\left( V^{'}_{ij}; G \right)$ and thus $h_{ij} \in W^{2,q}\left( V_{ij}; G \right),$ for \emph{any} $\frac{n}{2} < q < n.$ Sobolev embedding now proves the H\"{o}lder continuity of $h_{ij}$ in $V_{ij}.$ This proves the result.
\end{proof}
\subsection{Smooth approximation theorems: critical dimension}
\begin{theorem}[Smooth approximation of Sobolev bundles with Sobolev connection: critical case]\label{smoothaapproxcriticalconn}
	Let $kp=n.$ Given any $P \in \mathcal{P}_{G}^{k,p}\left(M^{n}\right)$, $A \in \mathcal{U}^{k,p}\left( P \right)$ and any $\varepsilon > 0,$ there is a smooth principal $G$-bundle $P^{\varepsilon} \in \mathcal{P}_{G}^{\infty}\left(M^{n}\right)$ with a smooth connection $A^{\varepsilon} \in \mathcal{A}^{\infty}\left( P^{\varepsilon} \right)$ such that $P^{\varepsilon}$ is $\varepsilon$-close to $P$ in the $W^{k,p}$ norm, $P^{\varepsilon} \stackrel{W^{k,p}}{\simeq_{\rho}} P $ and $A^{\varepsilon}$ is $\varepsilon$-close to the pullback of $A$ on $P^{\varepsilon}$ in $\mathcal{U}^{k,p}$ norm. \smallskip 
	
	\noindent More precisely, if $\left( P, A \right) = \left( \left\lbrace U_{\alpha}\right\rbrace_{\alpha \in I}, \left\lbrace g_{\alpha\beta} \right\rbrace_{\alpha, \beta \in I}, \left\lbrace A_{\alpha}\right\rbrace_{\alpha \in I} \right), $ then there exists a refinement $\left\lbrace V_{j}\right\rbrace_{j \in J}$ such that there exists 
	\textbf{smooth} Lie-algebra valued $1$-forms $A^{\varepsilon}_{j} \in C^{\infty}\left(V_{j}; \Lambda^{1}\mathbb{R}^{n}\otimes \mathfrak{g} \right)$ and \textbf{smooth} transition maps $h_{ij}\in C^{\infty}\left(V_{i}\cap V_{j}; G\right)  $ for all $i,j \in J,$ whenever the intersection is non-empty, satisfying 
	\begin{itemize}
		\item[(i)] $h_{ij}h_{jk} = h_{ik} \quad \text{ for a.e. } x \in V_{ijk} \text{ whenever } V_{ijk} \neq \emptyset, $
		\item[(ii)] $A^{\varepsilon}_{j} = h^{-1}_{ij}dh_{ij} + h^{-1}_{ij}A^{\varepsilon}_{i}h_{ij} \qquad \text{ for all } i,j \in J \text{ with } V_{ij} \neq \emptyset,$
		\item[(iii)] $ \left\lVert h_{ij} - g_{\phi(i)\phi(j)}\right\rVert_{W^{k,p}\left( V_{ij};G\right)} \leq \varepsilon \text{ whenever } V_{ijk} \neq \emptyset,$ 
		\item[(iv)] For each $j \in J,$ there exists maps $\rho_{j} \in W^{k,p} \left(  V_{j}; G \right) $ such that $$h_{ij} = \rho_{i}^{-1} g_{\phi(i)\phi(j)} \rho_{j}  \qquad \text{ for a.e. } x \in V_{ij} \text{ whenever } V_{ij} \neq \emptyset.$$ 
		\item[(v)] $ \left\lVert A^{\varepsilon}_{j} - A^{\rho_{j}}_{\phi(j)}\right\rVert_{\mathcal{U}^{k,p}\left( V_{j}; \Lambda^{1}\mathbb{R}^{n}\otimes \mathfrak{g} \right)} \leq \varepsilon,$  for every $j \in J,$ where $\phi: J \rightarrow I$ is the refinement map.  
	\end{itemize}
\end{theorem}
\begin{remark}
	(i) simply means that the cocycle data $\left( \left\lbrace V_{j}\right\rbrace_{j \in J}, \left\lbrace h_{ij} \right\rbrace_{i,j \in J} \right)$ defines a smooth bundle $P^{\varepsilon}\in \mathcal{P}_{G}^{\infty}\left(M^{n}\right) $. Conclusion (ii) is the requirement that the local representatives $A^{\varepsilon}_{j}$ actually defines a global connection form $A^{\varepsilon}$ on the bundle $P^{\varepsilon}.$ Conclusion (iii) and (iv), respectively, expresses the fact that $P$ and $P^{\varepsilon}$ are $\varepsilon$-close in the $W^{k,p}$ norm and are $W^{k,p}$-equivalent, i.e $P^{\varepsilon} \stackrel{W^{k,p}}{\simeq_{\rho}} P .$ Conclusion (v) means that on $P^{\varepsilon},$ the connection $A^{\varepsilon}$ is $\varepsilon$-close in the $\mathcal{U}^{k,p}$ norm to the pullback connection $\left(\rho\right)^{\ast}A = A^{\rho},$ obtained by pulling back the connection $A$ from $P$ to $P^{\varepsilon},$ by the bundle equivalence map $\rho.$   
\end{remark}
\begin{remark}\label{closeness to A vs pullback of A}
	One can also show the estimates \begin{equation*}
	\left\lVert A^{\varepsilon}_{j} - A_{\phi(j)}\right\rVert_{\mathcal{U}^{k,p}\left( V_{j}; \Lambda^{1}\mathbb{R}^{n}\otimes \mathfrak{g} \right)} \leq \varepsilon, \quad   \text{ for every }j \in J.
	\end{equation*}
	But these local estimates are somewhat meaningless since the collection of local $\mathfrak{g}$-valued $1$-forms $A_{\phi(j)}:V_{j} \rightarrow  \Lambda^{1}\mathbb{R}^{n}\otimes \mathfrak{g} ,$ $j \in J,$ \textbf{does not} in general define a \emph{connection form} on the bundle $P^{\varepsilon}.$ In other words, they \textbf{are not} in general the local expressions for a global $\mathfrak{g}$-valued $1$-form on $M^{n}$ in the bundle co-ordinates of $P^{\varepsilon}.$ The local representatives for the \emph{global} form $A: M^{n}\rightarrow \Lambda^{1}T^{\ast}M^{n}\otimes \mathfrak{g}$ in the bundle coordinates of $P^{\varepsilon}$ are precisely the forms $A^{\rho_{j}}_{\phi(j)}:V_{j} \rightarrow  \Lambda^{1}\mathbb{R}^{n}\otimes \mathfrak{g}, j \in J.$   
\end{remark}
\begin{proof}
	We prove only the case $k=1.$ The case $k = 2$ adds no essential new difficulties. We fix a representation  $ P = \left( \left\lbrace U_{\alpha}\right\rbrace_{\alpha \in I}, \left\lbrace g_{\alpha\beta} \right\rbrace_{\alpha, \beta \in I} \right) $ of our $W^{1,n}$ principal $G$-bundle and also assume that the connection form $A\in \mathcal{U}^{1,n}\left(P\right)$ is given by the local representatives $\left\lbrace A_{\alpha} \right\rbrace_{\alpha \in I},$ i.e. we have 
	\begin{align}
	A_{\alpha} \in L^{n}\left( U_{\alpha}; \Lambda^{1} \mathbb{R}^{n}\otimes \mathfrak{g}\right) \text{ and } dA_{\alpha} \in L^{\frac{n}{2}}\left( U_{\alpha}; \Lambda^{2} \mathbb{R}^{n}\otimes \mathfrak{g}\right),
	\end{align}
	and the gluing relations 
	\begin{equation}\label{gluing proof}
	A_{\beta} = g_{\alpha\beta}^{-1} d g_{\alpha\beta} + g_{\alpha\beta}^{-1} A_{\alpha} g_{\alpha\beta} \quad \text{ a.e. in } U_{\alpha} \cap U_{\beta}
	\end{equation} 
	holds whenever $U_{\alpha} \cap U_{\beta} \neq \emptyset.$ We divide the proof into several steps.\smallskip   
	
	\noindent \textbf{Step 1: Putting the connection in local Coulomb gauges} We choose a refinement $\left\lbrace V_{j}\right\rbrace_{j \in J}$ by small geodesic balls with the refinement map $\phi$ such that for every $ j \in J,$ we have 
	\begin{align}
	\left\lVert F_{A_{j}}\right\rVert_{L^{\frac{n}{2}}\left( V_{j};\Lambda^{2}\mathbb{R}^{n}\otimes \mathfrak{g} \right)} &< \min\left\lbrace \frac{\varepsilon_{Coulomb}}{16}, \frac{\varepsilon_{\Delta_{Cr}}}{64C_{Coulomb}}, \frac{\varepsilon}{64C_{Coulomb}C^{6}_{G}}   \right\rbrace \label{coulombsmallness}, \\ \left\lVert A_{j}\right\rVert_{L^{n}\left( V_{j};\Lambda^{1}\mathbb{R}^{n}\otimes \mathfrak{g} \right)} &<\frac{\varepsilon}{64C^{6}_{G}},  \label{Lnsmallness}
	\end{align}  where $A_{j}:= A_{\phi(j)}|_{V^{'}_{j}},$ and for every $ i, j \in J$ with $V_{ij}\neq \emptyset,$ 
	\begin{align}
	\left\lVert dg_{\phi(i)\phi(j)}\right\rVert_{L^{n}\left( V_{ij};\Lambda^{1}\mathbb{R}^{n}\otimes \mathfrak{g} \right)}, \left\lVert g_{\phi(i)\phi(j)}\right\rVert_{L^{n}\left( V_{ij}; G \right)} <\frac{\varepsilon}{64C^{6}_{G}}.   \label{gLnnormsmallness}
	\end{align} By lemma \ref{Coulombgauge}, applied to the small geodesic balls $V_{j},$ for each $j \in J, $ there exist maps $\rho_{j}: V_{j} \rightarrow G$ such that 
	 $d^{\ast} A_{j}^{\rho_{j}} = 0$  in $V_{j},$ $ \iota^{\ast}_{\partial V_{j}} \left( \ast A_{j}^{\rho_{j}}\right)  = 0$  on $\partial V_{j},$ and we have the estimates ( all norms on $V_{j}$), 
		\begin{align}\label{Coulombestimate}
		\left\lVert \nabla A_{j}^{\rho_{j}}\right\rVert_{L^{\frac{n}{2}}} + \left\lVert  A_{j}^{\rho_{j}}\right\rVert_{L^{n}}  &\leq C_{Coulomb}\left\lVert F_{A_{j}}\right\rVert_{L^{\frac{n}{2}}},\\
		\label{gauge estimate} 
		\left\lVert d\rho_{j}\right\rVert_{L^{n}}  &\leq C_{G}\left( C_{Coulomb}\left\lVert F_{A_{j}}\right\rVert_{L^{\frac{n}{2}}}  + \left\lVert A_{j}\right\rVert_{L^{n}}\right).
		\end{align}	
	
	\noindent \textbf{Step 2: Gluing local Coulomb gauges} Now we wish to show that the data $\left( \left\lbrace V_{j}\right\rbrace_{j \in J}, \left\lbrace h_{ij} \right\rbrace_{i, j \in J} \right):= P^{\varepsilon}_{A_{Coulomb}} $ defines a $W^{2,q}$ bundle for some $\frac{n}{2}< q < n,$ which is $\varepsilon$-close to $P$ in the $W^{1,n}$ norm and  on which the pullback of the original connection $A$ defines a $W^{1,\frac{n}{2}}$ connection in the Coulomb gauge, where \begin{align}\label{transferring to coulomb gauge}
	h_{ij} : = \rho_{i}^{-1}g_{\phi(i)\phi(j)}\rho_{j}:V^{'}_{ij} \rightarrow G \quad \text{ for every }i,j \in J \text{ with } V^{'}_{ij} \neq \emptyset.\end{align} It is easy to check that $h_{ij}$ are $W^{1,n}$ cocycles,  proving $P^{\varepsilon}_{A_{Coulomb}}\in \mathcal{P}_{G}^{1,n}\left(M^{n}\right).$ Clearly,  $A^{\rho}$ is a $W^{1,\frac{n}{2}}$ connection on $P^{\varepsilon}_{A_{Coulomb}}$ which is Coulomb.  Thus $W^{2,q}$ regularity of $P^{\varepsilon}_{A_{Coulomb}}$ follows by Theorem \ref{regularity Coulomb}.  Simple computation yields 
	\begin{align*}
	dh_{ij} - dg_{\phi(i)\phi(j)} =\begin{multlined}[t]
	d\rho^{-1}_{i}g_{\phi(i)\phi(j)}\rho_{j} + \rho_{i}^{-1} g_{\phi(i)\phi(j)}d\rho_{j} + \left(\rho_{i}^{-1} - \mathbf{1}_{G} \right) dg_{\phi(i)\phi(j)}\rho_{j} \\ +  dg_{\phi(i)\phi(j)}\left( \rho_{j} - \mathbf{1}_{G} \right)
	\end{multlined} 
	\end{align*}
	Hence, using \eqref{gauge estimate} and our choice in \eqref{coulombsmallness} and \eqref{Lnsmallness} and \eqref{gLnnormsmallness},
	\begin{align*}
	\left\lVert dh_{ij} - dg_{\phi(i)\phi(j)}\right\rVert_{L^{n}} \leq C^{2}_{G}\left( C^{2}_{G}\left\lVert d\rho_{i}\right\rVert_{L^{n}} + \left\lVert d\rho_{j}\right\rVert_{L^{n}}  + 4\left\lVert dg_{\phi(i)\phi(j)}\right\rVert_{L^{n}}\right) \leq \frac{\varepsilon}{4}.
	\end{align*}
	We also have, by \eqref{gLnnormsmallness},  
	\begin{align*}
	\left\lVert h_{ij} - g_{\phi(i)\phi(j)}\right\rVert_{L^{n}} =\left\lVert \left(\rho_{i}^{-1} - \mathbf{1}_{G} \right) g_{\phi(i)\phi(j)}\rho_{j}  +  g_{\phi(i)\phi(j)}\left( \rho_{j} - \mathbf{1}_{G} \right)\right\rVert_{L^{n}} \leq \frac{\varepsilon}{4}.
	\end{align*}
	Combing, we obtain the estimate 
	\begin{align}\label{Wkpclose}
	\left\lVert h_{ij} - g_{\phi(i)\phi(j)}\right\rVert_{W^{1,n}\left( V_{ij};G\right)} \leq \frac{\varepsilon}{2} \qquad \text{ whenever } V_{ij} \neq \emptyset.
	\end{align}
	Note that \eqref{Wkpclose} together with \eqref{gLnnormsmallness} implies 
	\begin{align}\label{Cobundlegaugesmallness}
	\left\lVert dh_{ij}\right\rVert_{L^{n}\left( V_{ij}; \Lambda^{1}\mathbb{R}^{n}\otimes\mathfrak{g}\right)} \leq \varepsilon \qquad \text{ for every } i,j \in J \text{ with } V_{ij} \neq \emptyset.
	\end{align}\smallskip   
	
	\noindent \textbf{Step 3: Approximation of connection} We pick an exponent $\frac{n}{2} < q < n.$ By step 2, we can assume, without loss of generality,  that we started with a $W^{1,\frac{n}{2}}$ connection $A$ in the Coulomb gauge  on a $W^{2,q}\cap C^{0}$-bundle $P = \left( \left\lbrace U_{\alpha}\right\rbrace_{\alpha \in I},\left\lbrace g_{\alpha\beta}\right\rbrace_{\alpha, \beta \in I} \right)$ and in view of \eqref{coulombsmallness}, \eqref{Coulombestimate} and \eqref{Cobundlegaugesmallness}, we have, 	\begin{align}
	&\left\lVert \nabla A_{\alpha}\right\rVert_{L^{\frac{n}{2}}\left( U_{\alpha};\mathbb{R}^{n \times n}\otimes \mathfrak{g} \right)} + \left\lVert  A_{\alpha}\right\rVert_{L^{n}\left( U_{\alpha} ;\Lambda^{1}\mathbb{R}^{n}\otimes \mathfrak{g} \right)} \leq \frac{\varepsilon}{64C^{6}_{G}} \quad \text{ for every } \alpha \in I, \label{smallnessofconnnormoriginal}\\
		 \label{newbundleguagesmall}&\left\lVert  dg_{\alpha\beta}\right\rVert_{L^{n}\left( U_{\alpha\beta} ;\Lambda^{1}\mathbb{R}^{n}\otimes \mathfrak{g} \right)} \leq \varepsilon  \quad \text{ for every } \alpha, \beta \in I \text{ with } U_{\alpha\beta} \neq \emptyset. 
	\end{align} By Theorem \ref{smoothingC0bundle}, there exists a smooth bundle $P^{\varepsilon}= \left( \left\lbrace V_{j}\right\rbrace_{j \in J},\left\lbrace g^{\varepsilon}_{ij}\right\rbrace_{i,j \in J} \right)$ which is $W^{2,q}$ equivalent to and $\varepsilon$-close to $P$ in $W^{2,q}$ norm. More precisely, there exists a refinement $\left\lbrace V_{j}\right\rbrace_{j \in J}$ of $\left\lbrace U_{\alpha}\right\rbrace_{\alpha \in I},$  with refinement map $\phi: J \rightarrow I,$  such that there exists \textbf{continuous} maps $\sigma_{j} \in W^{2,q}\left(V_{j}; G\right)$ and 
	\textbf{smooth} transition maps $g^{\varepsilon}_{ij}\in C^{\infty}\left(V_{i}\cap V_{j}; G\right)  $ for all $i,j \in J,$ whenever the intersection is non-empty, satisfying 
	\begin{itemize}
		\item[(i)] $g^{\varepsilon}_{ij}g^{\varepsilon}_{jk} = g^{\varepsilon}_{ik} \quad \text{ for a.e. } x \in V_{ijk} \text{ whenever } V_{ijk} \neq \emptyset, $
		\item[(ii)] $g^{\varepsilon}_{ij} = \sigma_{i}^{-1} g_{\phi(i)\phi(j)} \sigma_{j}  \qquad \text{ for a.e. } x \in V_{ij} \text{ whenever } V_{ij} \neq \emptyset.$ 
		\item[(iii)] $ \left\lVert g^{\varepsilon}_{ij} - g_{\phi(i)\phi(j)}\right\rVert_{W^{2,q}\left( V_{ij};G\right)} \leq \frac{\varepsilon}{64C^{6}_{G}} \text{ whenever } V_{ij} \neq \emptyset,$ 
		\item[(iv)] and, for every $j \in J,$ the estimates 
		\begin{equation}\label{sigma estimate}
		\left\lVert \sigma_{j} - \mathbf{1}_{G}\right\rVert_{L^{\infty}\left( V_{j}; G \right)} \leq \frac{\varepsilon}{64C^{6}_{G}} \quad \text{ and }\quad  \left\lVert d\sigma_{j} \right\rVert_{W^{1,q}\left( V_{j}; \Lambda^{1}\mathbb{R}^{n}\otimes\mathfrak{g} \right)} \leq \frac{\varepsilon}{64C^{6}_{G}}.
		\end{equation} 
		\end{itemize} 
		We note that (iv) and \eqref{newbundleguagesmall} together implies the estimate 
	\begin{align}\label{smoothbundlegaugebound}
	\left\lVert dg^{\varepsilon}_{ij} \right\rVert_{L^{n}\left( V_{ij};\Lambda^{1}\mathbb{R}^{n}\otimes\mathfrak{g}\right)} \leq \varepsilon \quad \text{ for every } i,j \in J \text{ with } V_{ij} \neq \emptyset.
	\end{align}
	Observe that the pullback of $A$ on $P^{\varepsilon}$ is a $W^{1,\frac{n}{2}}$ connection on $P^{\varepsilon}$. Indeed, denoting the local representatives of the pullback by  
	\begin{align}\label{definitionintermediateconn}
	\tilde{A}_{j}: = A_{j}^{\sigma_{j}} = \sigma_{j}^{-1}d\sigma_{j} + \sigma_{j}^{-1}A_{\phi(j)}\sigma_{j} \qquad \text{ in } V_{j} \qquad \text{ for each }j \in J,
	\end{align}
	we infer $\tilde{A}_{j}$ is $W^{1,\frac{n}{2}},$ since $A_{\phi(j)}$ is $W^{1,\frac{n}{2}}$ and $\sigma_{j}$ is $W^{2,q}$ for some $q > \frac{n}{2}.$
	Next we show that there is a smooth connection form $B= \left\lbrace B_{j}\right\rbrace_{j \in J}$ on $P^{\varepsilon}$ which is $\varepsilon$-close to the pullback of our original connection $A$ on $P^{\varepsilon}$ in the  $\mathcal{U}^{1,n}$ norm. Note that approximating $\tilde{A}_{j}$ by smooth forms is easy, but the real point, similar in spirit to Remark \ref{closeness to A vs pullback of A}, is to ensure that the approximating forms $B_{j}$ satisfy the gluing relations 
	\begin{align}\label{verificationgluingconn}
	B_{j}= \left( g_{ij}^{\varepsilon}\right)^{-1}dg^{\varepsilon}_{ij} + \left(g_{ij}^{\varepsilon}\right)^{-1}B_{i}g^{\varepsilon}_{ij} \qquad \text{ in } V_{ij} \text{ whenever } V_{ij} \neq \emptyset.
	\end{align}  We divide the proof into two substeps. \smallskip

	\noindent \textbf{Step 3a: Construction of approximating forms}	We choose a partition of unity $\left\lbrace \psi_{j}\right\rbrace_{j \in J}$ subordinate to the cover  $\left\lbrace V_{j}\right\rbrace_{j \in J}$ such that  we have the bounds 
	\begin{align}\label{partionofunitychoice}
	\left\lVert d\psi_{l}\right\rVert_{L^{\infty}\left( V_{lj} \right)} \leq \frac{C_{part}}{\operatorname{meas}\left(V_{lj}\right)^{\frac{1}{n}}}\qquad \text{ for every } j,l \in J \text{ with } V_{lj} \neq \emptyset. 
	\end{align} where  $C_{part} \geq 1$ is a fixed constant. By density, we can find, for each $j \in J,$ $\tilde{A}^{\varepsilon}_{j} \in C^{\infty}\left( V_{j}; \Lambda^{1}\mathbb{R}^{n}\otimes \mathfrak{g}\right)$  such that 
	\begin{align}\label{smoothingconnapproxchoice}
	\left\lVert \tilde{A}^{\varepsilon}_{j}- \tilde{A}_{j}\right\rVert_{W^{1,\frac{n}{2}}\left( V_{j}; \Lambda^{1}\mathbb{R}^{n}\otimes \mathfrak{g}\right)} \leq \frac{\varepsilon}{64\left( C_{part}N_{J} \right)^{2}C^{6}_{G}}, 
	\end{align}
	where $N_{J}=\#J$ denote the cardinality of the finite index set $J.$  We define 
	\begin{align}\label{definitionapproxconn}
	B_{j}: = \sum_{\substack{l \in J \\ V_{jl} \neq \emptyset}} \psi_{l}\left[ \left( g_{lj}^{\varepsilon}\right)^{-1}dg^{\varepsilon}_{lj} + \left(g_{lj}^{\varepsilon}\right)^{-1}\tilde{A}^{\varepsilon}_{l}g^{\varepsilon}_{lj}\right].
	\end{align}
	Note that the possibility $j=l$ is not excluded. Clearly, $B_{j}$ is smooth for each $j \in J.$ By a straight forward computation using the identity $ dg^{\varepsilon}_{lj} -  dg^{\varepsilon}_{li}g^{\varepsilon}_{ij} = g_{li}^{\varepsilon}dg^{\varepsilon}_{ij},$ obtained by differentiating the cocycle condition $g^{\varepsilon}_{lj}=g_{li}^{\varepsilon} g^{\varepsilon}_{ij},$ we deduce \begin{align*}
	g^{\varepsilon}_{ij}B_{j} - B_{i}g^{\varepsilon}_{ij} =\left(\psi_{i} + \psi_{j}\right)dg^{\varepsilon}_{ij} + \sum_{\substack{l \in J,\\ l\neq i,j,\\V_{ijl} \neq \emptyset }} \psi_{l} dg^{\varepsilon}_{ij}. 
	\end{align*}
	Since $\left\lbrace \psi_{j} \right\rbrace_{j}$ is a partition of unity, this proves prove \eqref{verificationgluingconn}. \smallskip 
	
	\noindent\textbf{Step 3b: Approximation bounds} It only remains to estimate  $\left\lVert dB_{j} - d\tilde{A}_{j} \right\rVert_{L^{\frac{n}{2}}} $ and $\left\lVert B_{j} - \tilde{A}_{j} \right\rVert_{L^{n}}.$ Observe that since $\tilde{A}$ is a connection on the bundle $P^{\varepsilon}$, from the gluing relations and by properties of a partition of unity, we can write  
	\begin{align}\label{Atilde in partition of unity}
	\tilde{A}_{j} = \sum_{\substack{l \in J \\ V_{jl} \neq \emptyset}} \psi_{l} \left( \tilde{A}_{j}\big|_{V_{jl}} \right)  = \sum_{\substack{l \in J \\ V_{jl} \neq \emptyset}} \psi_{l}\left[  \left(g^{\varepsilon}_{lj}\right)^{-1}dg^{\varepsilon}_{lj} + \left(g^{\varepsilon}_{lj}\right)^{-1}\tilde{A}_{l}
	g^{\varepsilon}_{lj}\right].
	\end{align} 
	Subtracting \eqref{Atilde in partition of unity} from \eqref{definitionapproxconn}, we can estimate $\left\lVert B_{j} - \tilde{A}_{j} \right\rVert_{L^{n}}.$ 
	Next, we compute 
	\begin{align*}
	dB_{j}&- d\tilde{A}_{j} \\&= \sum_{\substack{l \in J \\ V_{jl} \neq \emptyset}} \left[ d\psi_{l}\wedge  \left(g_{lj}^{\varepsilon}\right)^{-1}\left(\tilde{A}^{\varepsilon}_{l} -\tilde{A}_{l} \right)g^{\varepsilon}_{lj} + \psi_{l} d \left( \left(g_{lj}^{\varepsilon}\right)^{-1}\tilde{A}^{\varepsilon}_{l}g^{\varepsilon}_{lj} -\left(g_{lj}^{\varepsilon}\right)^{-1}\tilde{A}_{l}g^{\varepsilon}_{lj} \right)\right].
	\end{align*}
The second term is easy to estimate. For the first, we have, for fixed $l,j \in J,$  
\begin{align*}
\left\lVert d\psi_{l}\wedge  \left(g_{lj}^{\varepsilon}\right)^{-1}\left(\tilde{A}^{\varepsilon}_{l} -\tilde{A}_{l} \right)g^{\varepsilon}_{lj} \right\rVert_{L^{\frac{n}{2}}\left( V_{lj}\right)} &\stackrel{\eqref{partionofunitychoice}}{\leq} \frac{C_{part}}{\operatorname{meas}\left(V_{lj}\right)^{\frac{1}{n}}}\cdot C^{2}_{G} \left\lVert \tilde{A}^{\varepsilon}_{l} -\tilde{A}_{l} \right\rVert_{L^{\frac{n}{2}}\left( V_{lj}\right)} \\
&\stackrel{\text{H\"{o}lder}}{\leq}
C_{part}C^{2}_{G} \left\lVert \tilde{A}^{\varepsilon}_{l} -\tilde{A}_{l} \right\rVert_{L^{n}\left( V_{lj}\right)} \stackrel{\eqref{smoothingconnapproxchoice}}{\leq }\frac{\varepsilon}{64N^{2}_{J}}.
\end{align*}
Summing over $l \in J$ with $V_{jl} \neq \emptyset$ and setting $B= A^{\varepsilon},$ the proof is complete. \end{proof}
\noindent The following result is a consequence, which immediately implies Theorem \ref{cocycle smoothing}.   
\begin{theorem}[Smooth approximation of Sobolev bundles: critical case]\label{smoothaapproxcritical}
	Let $kp=n.$ Given any $P \in \mathcal{P}_{G}^{k,p}\left(M^{n}\right)$ and any $0 < \varepsilon < 1,$ there is a smooth principal $G$-bundle $P^{\varepsilon} \in \mathcal{P}_{G}^{\infty}\left(M^{n}\right)$ which is $\varepsilon$-close to $P$ in the $W^{k,p}$ norm such that $P^{\varepsilon} \stackrel{W^{k,p}}{\simeq} P .$ More precisely, if $P = \left( \left\lbrace U_{\alpha}\right\rbrace_{\alpha \in I}, \left\lbrace g_{\alpha\beta} \right\rbrace_{\alpha, \beta \in I} \right),$ then there exists a good refinement $\left\lbrace V_{j}\right\rbrace_{j \in J}$ of $\left\lbrace U_{\alpha}\right\rbrace_{\alpha \in I}$ such that there exists \textbf{smooth} transition maps $h_{ij}\in C^{\infty}\left(V_{i}\cap V_{j}; G\right)  $ for all $i,j \in J,$ whenever the intersection is non-empty, satisfying 
	\begin{itemize}
		\item[(i)] $h_{ij}h_{jk} = h_{ik} \quad \text{ for a.e. } x \in V_{ijk} \text{ whenever } V_{ijk} \neq \emptyset, $
		\item[(ii)] $ \left\lVert h_{ij} - g_{\phi(i)\phi(j)}\right\rVert_{W^{k,p}\left( V_{ij};G\right)} \leq \varepsilon \text{ whenever } V_{ijk} \neq \emptyset,$ where $\phi: J \rightarrow I$ is the refinement map. 
		\item[(iii)] For each $j \in J,$ there exists maps $\rho_{j} \in W^{k,p} \left(  V_{j}; G \right) $ such that $$h_{ij} = \rho_{i} g_{\phi(i)\phi(j)} \rho_{j}^{-1}  \qquad \text{ for a.e. } x \in V_{ij} \text{ whenever } V_{ij} \neq \emptyset.$$  
	\end{itemize}
\end{theorem}
\begin{proof}
	Once again we shall prove the case $k=1.$ The result will be an immediate corollary of theorem \ref{smoothaapproxcriticalconn} as soon as we show the following claim. 
		\begin{claim} Every $W^{1,n}$ bundle admits a $\mathcal{U}^{1,n}$ connection.	\end{claim} 
\noindent Pick a partition of unity $\left\lbrace \psi_{\alpha} \right\rbrace_{\alpha}$ subordinate to the cover $\left\lbrace U_{\alpha}\right\rbrace_{\alpha \in I}$ and define 
\begin{equation}\label{definingconn}
A_{\alpha} := \sum_{\substack{\beta \in I ,\beta \neq \alpha, \\ U_{\alpha}\cap U_{\beta} \neq \emptyset}} \psi_{\beta} g_{\beta\alpha}^{-1} dg_{\beta\alpha}  \qquad \text{ for each } \alpha \in I.
\end{equation}
Now clearly $ A_{\alpha} \in L^{n}$ and we also have $dA_{\alpha} \in L^{\frac{n}{2}},$ since \begin{align*}
\left\lVert d \left(  g_{\beta\alpha}^{-1} d g_{\beta\alpha} \right) \right\rVert_{L^{\frac{n}{2}}}= \left\lVert - g_{\beta\alpha}^{-1} dg_{\beta\alpha} \wedge g_{\beta\alpha}^{-1} dg_{\beta\alpha}\right\rVert_{L^{\frac{n}{2}}} \stackrel{\text{H\"{o}lder}}{\leq} C_{G}^{2}\left\lVert d g_{\beta\alpha}  \right\rVert_{L^{n}}^{2}.
\end{align*}
By a straight forward computation, we can verify the gluing condition \begin{equation*}
A_{\beta} = g_{\alpha\beta}^{-1} d g_{\alpha\beta} + g_{\alpha\beta}^{-1} A_{\alpha} g_{\alpha\beta} \qquad \text{ a.e. in } U_{\alpha} \cap U_{\beta}
\end{equation*} holds for every $\alpha, \beta \in I$ whenever $U_{\alpha} \cap U_{\beta} \neq \emptyset.$  This completes the proof. \end{proof}
\begin{remark}\label{possiblydifftopology_smoothapprox}
	Note that the smooth bundles $P^{\varepsilon},$ given by Theorem \ref{smoothaapproxcriticalconn} also satisfies the conclusions of Theorem \ref{smoothaapproxcritical} for any given $\mathcal{U}^{k,p}$ connection on $P.$  Indeed, we are deducing Theorem \ref{smoothaapproxcritical} from Theorem \ref{smoothaapproxcriticalconn} by showing the existence of \emph{one} such  connection. It is not hard to show that given a fixed connection $A \in \mathcal{U}^{k,p}\left( P \right)$, the topological isomorphism class of the smooth bundles $P^{\varepsilon}$, constructed in Theorem \ref{smoothaapproxcriticalconn} would be independent of $\varepsilon$ for $\varepsilon> 0$ small enough ( see section \ref{topology section} ). So, given a pair $\left( P, A \right)$, the approximating smooth bundles can be chosen to have a fixed topological isomorphism class. However, if $A, B \in \mathcal{U}^{k,p}\left( P \right)$ are two different $\mathcal{U}^{k,p}$ connection on $P,$ there is no reason for the approximating smooth bundles given by Theorem \ref{smoothaapproxcriticalconn} for the pair $\left(P, A\right)$ and $\left(P, B\right)$ to be $C^{0}$-equivalent. In general, they are not. In particular, this implies that given a bundle $P \in \mathcal{P}_{G}^{k,p}\left(M^{n}\right),$ it might be possible to construct two different sequence of smooth bundles $\left\lbrace P_{1}^{\varepsilon} \right\rbrace_{\varepsilon > 0}$ and $\left\lbrace P_{2}^{\varepsilon} \right\rbrace_{\varepsilon > 0}$, both of which approximates $P$ in the sense of Theorem \ref{smoothaapproxcritical}, but $\left[ P_{1}^{\varepsilon} \right]_{0} \neq \left[ P_{2}^{\varepsilon} \right]_{0}, $ for every $\varepsilon > 0.$  
\end{remark}

\subsection{Circle bundles in arbitrary dimension}
The results improve substantially if the $G$ is Abelian, since in this case finding a Coulomb gauge is much easier.  
\begin{theorem}\label{Abeliansmoothaapproxsupercritical}
	 Given any $P \in \mathcal{P}_{S^{1}}^{k,p}\left(M^{n}\right)$ $A \in \mathcal{U}^{k,p}\left( P \right)$ and any $\varepsilon > 0,$ there is a smooth principal $S^{1}$-bundle $P^{\varepsilon} \in \mathcal{P}_{S^{1}}^{\infty}\left(M^{n}\right)$ with a smooth connection $A^{\varepsilon} \in \mathcal{A}^{\infty}\left( P^{\varepsilon} \right)$ such that $P^{\varepsilon}$ is $\varepsilon$-close to $P$ in the $W^{k,p}$ norm, $A^{\varepsilon}$ is $\varepsilon$ close to $A$ in $\mathcal{U}^{k,p}$ norm and $P^{\varepsilon} \stackrel{W^{k,p}}{\simeq} P .$  More precisely, if $\left( P, A \right) = \left( \left\lbrace U_{\alpha}\right\rbrace_{\alpha \in I}, \left\lbrace g_{\alpha\beta} \right\rbrace_{\alpha, \beta \in I}, \left\lbrace A_{\alpha}\right\rbrace_{\alpha \in I} \right) ,$ then there exists a good refinement $\left\lbrace V_{j}\right\rbrace_{j \in J}$  such that there exists 
	\textbf{smooth} $i\mathbb{R}$-valued $1$-forms $B_{j} \in C^{\infty}\left(V_{j}; \Lambda^{1}\mathbb{R}^{n}\otimes i\mathbb{R} \right)$ and \textbf{smooth} transition maps $h_{ij}\in C^{\infty}\left(V_{i}\cap V_{j}; S^{1}\right)  $ for all $i,j \in J,$ whenever the intersection is non-empty, satisfying 
	\begin{itemize}
		\item[(i)] $h_{ij}h_{jk} = h_{ik} \quad \text{ for a.e. } x \in V_{ijk} \text{ whenever } V_{ijk} \neq \emptyset, $
		\item[(ii)] $B_{j} = h^{-1}_{ij}dh_{ij} + h^{-1}_{ij}B_{i}h_{ij} \qquad \text{ for all } i,j \in J \text{ with } V_{ij} \neq \emptyset,$
		\item[(iii)] $ \left\lVert h_{ij} - g_{\phi(i)\phi(j)}\right\rVert_{W^{k,p}\left( V_{ij};S^{1}\right)} \leq \varepsilon \text{ whenever } V_{ijk} \neq \emptyset,$ where $\phi: J \rightarrow I$ is the refinement map.
		\item[(iv)] For each $j \in J,$ there exists maps $\rho_{j} \in W^{k,p} \left(  V_{j}; S^{1} \right) $ such that $$h_{ij} = \rho_{i}^{-1} g_{\phi(i)\phi(j)} \rho_{j}  \qquad \text{ for a.e. } x \in V_{ij} \text{ whenever } V_{ij} \neq \emptyset.$$  
		\item[(v)]  $ \left\lVert B_{j} - A^{\rho_{j}}_{\phi(j)}\right\rVert_{\mathcal{U}^{k,p}\left( V_{j}; \Lambda^{1}\mathbb{R}^{n}\otimes i\mathbb{R} \right)} \leq \varepsilon,$  for every $j \in J.$ 
	\end{itemize}
\end{theorem}
\begin{proof}
	We follow the same approach as in the proof of theorem \ref{smoothaapproxcriticalconn}. The effectiveness of the approach via local Coulomb gauges makes the proof of this result quite easy. Firstly, since $S^{1}$ is Abelian, local Coulomb gauges always exist without the need for any smallness condition on the norm of the curvature. 
	We choose a good refinement $\left\lbrace V_{j}\right\rbrace_{j \in J}$ of $\left\lbrace U_{\alpha}\right\rbrace_{\alpha \in I}$ in such a way that there is an enlarged cover $\left\lbrace V^{'}_{j}\right\rbrace_{j \in J}$ which is also a refinement of $\left\lbrace U_{\alpha}\right\rbrace_{\alpha \in I}$ with the same refinement map $\phi: J \rightarrow I.$ More precisely, this means we have $$ V_{j} \subset\subset V^{'}_{j} \subset U_{\phi(j)} \quad \text{ for every } j \in J, \qquad \bigcup\limits_{j \in J} V_{j} = \bigcup\limits_{\alpha \in I} U_{\alpha} = M^{n}. $$ Now existence of local Coulomb gauges $\rho_{j}$ on $V^{'}_{j}$ boils down to finding a real-valued function $\psi_{j} \in W^{1,p}\left(V^{'}_{j}\right)$  solving the following inhomogeneous Neumann problem for the Laplacian
	\begin{align*}
	\left\lbrace \begin{aligned}
	\Delta \psi_{j} &= - i d^{\ast}A_{\phi(j)} &&\text{ in } V^{'}_{j}, \\
	\frac{\partial \psi_{j}}{\partial \nu} &= - i \iota^{\ast}_{\partial V^{'}_{j}} \left( \ast A_{\phi(j)}\right) &&\text{ on } \partial V^{'}_{j} .
	\end{aligned}\right. 
	\end{align*}
	Clearly, if $\psi_{j}$ solves the Neumann problem above, then $\rho_{j} = e^{id\psi_{j}}$ is the desired Coulomb gauge. But since $d^{\ast} A_{\phi(j)} = \ast \left[ d \left( \ast A_{\phi(j)}\right)\right]$, Stokes theorem implies the compatibility condition $$ - \int_{\partial V^{'}_{j}}\iota^{\ast}_{\partial V^{'}_{j}} \left( \ast A_{\phi(j)}\right) = - \int_{V^{'}_{j}} d^{\ast}A_{\phi(j)} . $$ Thus existence and estimates follow from standard elliptic theory. Gluing the local Coulomb gauges can be done exactly as before and thus we can construct gauges $h_{ij} = \rho_{i}^{-1}g_{\phi(i)\phi(j)}\rho_{j}$ such that $\left( \left\lbrace V_{j} \right\rbrace_{j \in J},  \left\lbrace h_{ij} \right\rbrace_{i,j \in J}\right)$ is a principal $S^{1}$-bundle of class $W^{1,p}$ over $M^{n}$ and the gluing relations for the connection are satisfied in $V^{'}_{ij}.$ But since  $S^{1}$ is Abelian, the gluing relations are 
	\begin{align}\label{abeliangluing}
	h_{ij}^{-1}dh_{ij} = A^{\rho_{j}}_{j} -A^{\rho_{i}}_{i} \qquad \text{ in } V^{'}_{ij} \text{ whenever } V^{'}_{ij} \neq \emptyset. 
	\end{align}   
	But $h_{ij}:V^{'}_{ij} \rightarrow S^{1} $ is a $W^{1,p}$ map and thus, by results in \cite{BethuelZheng_density_spheretargets} ( also \cite{BourgainBrezisMironescu_lifting} ), there exists a `lift' $\eta_{ij} \in W^{1,p}\left(V^{'}_{ij} \right)$ such that $ h_{ij} = e^{i\eta_{ij}}.$ Substituting in \eqref{abeliangluing}, we obtain 
	\begin{align*}
	d\eta_{ij} = i \left( A^{\rho_{j}}_{j} -A^{\rho_{i}}_{i}\right) \qquad \text{ in } V^{'}_{ij} \text{ whenever } V^{'}_{ij} \neq \emptyset. 
	\end{align*}
	Using the fact that both $d^{\ast}A^{\rho_{j}}_{j}$ and $d^{\ast}A^{\rho_{i}}_{i}$ are zero, we deduce 
	\begin{align*}
	\Delta \eta_{ij} = d^{\ast}d\eta_{ij} = 0 \qquad \text{ in } V^{'}_{ij} \text{ whenever } V^{'}_{ij} \neq \emptyset. 
	\end{align*}
	Thus, $\eta_{ij}$ is harmonic in $V^{'}_{ij}.$ By standard interior regularity for harmonic functions, $\eta_{ij}$ is actually \textbf{smooth} in $V_{ij}$ and consequently, so is $h_{ij}.$ Thus, we have smoothed the bundle in one step. Now we can follow step 3 of the proof of theorem \ref{smoothaapproxcriticalconn} and approximate the connection by smooth ones. The only difference is that the corresponding estimates are far simpler due to the fact that $S^{1}$ is Abelian. This completes the proof. \end{proof}

\section{Topology of bundles in the critical dimension}\label{topology section}
\subsection{Coulomb bundles and gauge transformations}
We have already shown in Theorem \ref{regularity Coulomb} that Coulomb bundles are $C^{0}$ bundles. Now we show their $C^{0}$-equivalence class is stable under $W^{k,p}$ gauge transformations for $kp =n.$ 
\begin{proposition}\label{gaugerelatedcoulomb}
	Let $kp=n.$ Let $P^{i} \in \mathcal{P}_{G}^{k,p}\left(M^{n}\right)$ and $A_{i} \in \mathcal{U}^{k,p}\left(P^{i}\right)$ such that $A^{i}$ is Coulomb on $P^{i}$, for $i=1,2,$ and $\left( P^{1}, A^{1}\right)\stackrel{W^{k,p}}{\simeq}_{\sigma} \left( P^{2}, A^{2}\right).$ Then $P_{1}$ and $P_{2}$ are $C^{0}$-equivalent. 	
\end{proposition}
\begin{proof}
	The proof is very similar to how we proved the continuity of Coulomb bundles, so we provide only a brief sketch. Since the connections are gauge related, we have
	$$ d\sigma_{i} = \sigma A_{i}^{2} - A_{i}^{1}\sigma \qquad \text{ in } U_{i},$$ 
	for every $i\in I. $ Since $A^{1}, A^{2}$ are both Coulomb, we have, 
	$$ - \Delta\sigma_{i} = \ast \left[ d\sigma_{i} \wedge \left( \ast A^{2}_{i}\right)\right] + \ast \left[  \left( \ast A^{1}_{i}\right) \wedge d\sigma_{i} \right]  \qquad \text{ in } U_{i},$$ for every $i\in I. $ But once again this is exactly of the form of eq \eqref{critical elliptic eqn} with $f=0.$ Thus, by passing to a refinement of the cover  such that $L^{n}$ norms of $A_{i}^{1}$ and $A_{i}^{2}$ are suitably small, using lemma \ref{ellipticCritical} we deduce the continuity of $\sigma_{i}$ in the interior. The proof is concluded by slightly shrinking the domains.    
\end{proof}
From this we deduce the uniqueness of Coulomb bundles for a connection. 
\begin{proposition}[Uniqueness of Coulomb bundles]
Let $kp=n.$ Given a pair $\left( P, A \right),$ where $P \in \mathcal{P}_{G}^{k,p}\left(M^{n}\right)$ and $A \in \mathcal{U}^{k,p}\left(P\right)$ there exists a  $C^{0}$-bundle $P_{A_{coulomb}}$ , \emph{unique up to $C^{0}$-equivalence}, such that $P_{A_{coulomb}} \stackrel{W^{k,p}}{\simeq}_{\sigma} P$ and $\sigma^{\ast}A$ is a $W^{1,\frac{n}{2}}$ connection on $P^{\varepsilon}_{A_{coulomb}}$ which is Coulomb.
\end{proposition}
\begin{proof}
 Given any $\varepsilon > 0, $ the bundle $P^{\varepsilon}_{A_{coulomb}} \in \mathcal{P}^{0}\left(M^{n}\right)$ constructed in step 1 and 2 of the proof of theorem \ref{smoothaapproxcriticalconn} is one such bundle, so we only have to prove the uniqueness claim. But if $P_{1}, P_{2}$ be any two such bundles, then they are gauge-related to $P$ by $W^{k, p}$ gauges $\sigma_{1}, \sigma_{2}$ respectively. Then the pairs $\left( P_{1}, \sigma_{1}^{\ast}A \right),$ and $\left( P_{2}, \sigma_{2}^{\ast}A \right)$ are themselves gauge related by $\sigma_{1}^{-1}\circ \sigma_{2}.$ Thus Proposition \ref{gaugerelatedcoulomb} concludes the proof.  	
\end{proof}
\subsection{Definition of the topology}
\begin{definition}\label{definition of topology}\textbf{(Topology for bundles with connection)} Let $kp=n.$ Then to each pair $\left( P, A \right),$ where $P \in \mathcal{P}_{G}^{k,p}\left(M^{n}\right)$ and $A \in \mathcal{U}^{k,p}\left(P\right)$, we can associate a topological isomorphism class, denoted by $\left[P_{A} \right]_{W^{k,p}}$ by 
	\begin{align*}
	\left[ P_{A} \right]_{W^{k,p}}:= \left[ P_{A_{coulomb}} \right]_{C^{0}}. 
	\end{align*}
\end{definition}
\begin{remark}\label{specific choice of connection}
	Note that the topological isomorphism class is associated to a pair $(P, A)$ and \emph{not} to the bundle alone. However, the $\mathcal{U}^{1,n}$ connection we constructed in the proof of Theorem \ref{smoothaapproxcritical} is in some sense `canonical', i.e. it is constructed out of a partition of unity for the cover and the bundle transition maps only. So one can chose to always use this connection and thereby associate a topological isomorphism class to the bundle $P$ alone. This is what seems to be what Shevchishin has done implicitly in \cite{Shevchishin_limitholonomyYangMills}. However, in our opinion, such an assignment is undesirable for two reasons. Firstly, it once again decouples the topological information of the bundle from the connection and thus defeats the purpose of tying the topology of the bundles to the analysis of connections and curvatures under Yang-Mills energy. The second reason is that from the point of view of geometry, philosophically there is no canonical choice for a connection on a principal bundle.    
\end{remark}
The topological isomorphism class we defined is stable under $W^{k,p}$-gauge transformations.  Since by the transitivity of $W^{k,p}$ gauge relations, the associated Coulomb bundles are gauge related, this follows from Proposition \ref{gaugerelatedcoulomb}.  
\begin{proposition}[Stability of topology under gauge transformation]
if $kp=n,$  $P^{i} \in \mathcal{P}_{G}^{k,p}\left(M^{n}\right)$ and $A^{i} \in \mathcal{U}^{k,p}\left(P^{i}\right)$ for $i=1,2,$ and $\left( P^{1}, A^{1}\right)\stackrel{W^{k,p}}{\simeq}_{\sigma} \left( P^{2}, A^{2}\right).$ then $ \left[ P^{1}_{A^{1}} \right]_{W^{k,p}} = \left[ P^{2}_{A^{2}} \right]_{W^{k,p}}.$
\end{proposition} 

\subsection{Compatibility for regular bundles and connections}
As we remarked before, the topological isomorphism class is associated to the \emph{pair} $\left(P, A\right)$ and not a property of $P$ alone. This is in sharp contrast to the case of Sobolev bundles in the subcritical regime $kp > n,$ where the transition functions of the bundle alone, being continuous, are sufficient to determine the topology of the bundle. Here, on the other hand, we encode the topological information about the bundle in the connection. Indeed, if $P \in \mathcal{P}_{G}^{k,p}\left(M^{n}\right)$ be a $W^{k,p}$ bundle over $M^{n}$ and $A, B \in \mathcal{U}^{k,p}\left(P\right)$ are two different  $\mathcal{U}^{k,p}$ connection on $P$ with $kp=n,$ then  in general $\left[ P_{A} \right]_{W^{k,p}} \neq  \left[ P_{B} \right]_{W^{k,p}}.$ Also, even if $P$ is a smooth bundle, then in general $\left[ P \right]_{C^{0}} \neq  \left[ P_{A_{coulomb}} \right]_{C^{0}}$ for $A \in \mathcal{U}^{k,p}\left(P\right)$ with $kp=n.$  However, we should justifiably demand that for a smooth connection on a smooth bundle, the notion of topology defined here should coincide with the usual notion. Now we are going to show that this is indeed the case. In fact, we shall show that the only relevant factor here is the regularity of the connection. 
\begin{theorem}\label{topology notions coincide}
Let $kp=n.$ Let $P \in \mathcal{P}_{G}^{k,p}\left(M^{n}\right) \cap \mathcal{P}_{G}^{0}\left(M^{n}\right)$ and let $A \in \mathcal{U}^{k,p}\left(P\right).$ Then $P_{A_{Coulomb}}$ and $P$ are $C^{0}$-equivalent as soon as $d^{\ast}A \in L^{\left(\frac{n}{2},1\right)}.$ \end{theorem}
\begin{remark} In particular, $\left[ P_{A} \right]_{W^{k,p}} =  \left[ P \right]_{C^{0}},$  for any \textbf{smooth} connections $A$ on a $W^{k,p}\cap C^{0}$ bundle $P$ with $kp=n.$ 
\end{remark}
\begin{proof}
	We have the following equations for the Coulomb gauges for $A.$
	\begin{align*}
	\text{ for every } i \in I, \qquad \left\lbrace \begin{aligned}
	&A_{i}^{\rho_{i}} = \rho_{i}^{-1}d\rho_{i} + \rho_{i}^{-1}A_{i}\rho_{i} &&\text{ in } U_{i}, \\
	&d^{\ast}\left( A_{i}^{\rho_{i}}\right) = 0 &&\text{ in } U_{i}.
	\end{aligned}\right. 
	\end{align*}
	From this we can deduce the equation 
	\begin{align*}
	-\Delta \rho_{i} = \ast \left[ d \rho_{i} \wedge \left( \ast A^{\rho_{i}}_{i}\right)\right] + \ast \left[ \left( \ast A_{i}\right)\wedge d \rho_{i}  \right]- \left( d^{\ast} A_{i}\right) \rho_{i}   \qquad \text{ in } U_{i},  \text{ for every } i \in I.
	\end{align*}
	The last term on the right is $L^{\left(\frac{n}{2},1\right)}.$ By shrinking the domains to ensure smallness of the $L^{n}$ norms, Lemma \ref{ellipticCritical} concludes the proof by the Peetre-Sobolev embedding $W^{2,\left(\frac{n}{2},1\right)} \hookrightarrow W^{1,\left(n,1\right)}$ combined with the Stein  \cite{Stein_steintheorem} embedding $ W^{1,\left(n,1\right)}  \hookrightarrow C^{0}.$\end{proof}
\begin{theorem}\label{independent of connection}
	Let $kp=n.$ Let $P \in \mathcal{P}_{G}^{k,p}\left(M^{n}\right)$ and let $A, B \in \mathcal{U}^{k,p}\left(P\right).$ Then $P_{A_{Coulomb}}$ and $P_{B_{Coulomb}}$ are $C^{0}$-equivalent as soon as $d^{\ast}A, d^{\ast}B \in L^{\left(\frac{n}{2},1\right)}.$ 
	\end{theorem}

\begin{proof}
	We shall denote by $\rho$ and $\sigma$, the Coulomb gauges for $A$ and $B$ respectively. Since $P_{A_{Coulomb}}$ and $P_{B_{Coulomb}}$ are gauge related to $P$ by $\rho$ and $\sigma$ respectively, $P_{A_{Coulomb}}$ is gauge related to $P_{B_{Coulomb}}$ by the gauges $u= \sigma^{-1}\rho.$ By passing to a common refinement if necessary, we can assume that there exists a fixed cover $\left\lbrace U_{i} \right\rbrace_{i \in I}$ of $M^{n}$ such that for each $i \in I, $ we have, 
	\begin{align*}
	&A_{i}^{\rho_{i}} = \rho_{i}^{-1}d\rho_{i} + \rho_{i}^{-1}A_{i}\rho_{i} &&\text{ in } U_{i}, \\
	&d^{\ast}\left( A_{i}^{\rho_{i}}\right) = 0 &&\text{ in } U_{i}, \intertext{ and } 
	&B_{i}^{\sigma_{i}} = \sigma_{i}^{-1}d\sigma_{i} + \sigma_{i}^{-1}B_{i}\sigma_{i} &&\text{ in } U_{i}, \\
	&d^{\ast}\left( B_{i}^{\sigma_{i}}\right) = 0 &&\text{ in } U_{i}. 
	\end{align*}
	Arguing exactly as in Theorem \ref{topology notions coincide}, we see that up to shrinking the domains, both $\sigma$ and $\rho$ are continuous and hence so is $u.$ 	 
\end{proof} 
\subsection{Topology in the limit without curvature concentration}
As we have already seen, the notion of the topology we defined depends heavily on the regularity of the connection. In particular, given a fixed $W^{k,p}$ bundle $P$ over $M^{n}$ and a sequence of $\mathcal{U}^{k,p}$ connections $\left\lbrace A^{\nu}\right\rbrace_{\nu \geq 1} \subset \mathcal{U}^{k,p}\left( P\right),$ the associated topological isomorphism classes $\left[ P_{A^{\nu}}\right]_{W^{k,p}}$ can all be different when $kp=n.$ However, this can not happen if we have some information regarding the curvatures of the connections. This is the content of the following theorem. 
\begin{theorem}\label{equiintegrability theorem}
Let $kp=n.$ Let $\left\lbrace P^{\nu} \right\rbrace_{\nu \geq 1} \subset  \mathcal{P}_{G}^{k,p}\left(M^{n}\right)$ be such that there exists a common refinement $\left\lbrace W_{\alpha}\right\rbrace_{\alpha \in J}$ for the associated covers. Let	$ A^{\nu} \in \mathcal{U}^{k,p}\left( P^{\nu}\right)$ for all $\nu \geq 1$ such that $\left\lVert F_{A^{\nu}}\right\rVert_{L^{\frac{n}{2}}\left(M^{n}; \Lambda^{2}T^{\ast}M^{n}\otimes \mathfrak{g} \right)}$ is uniformly bounded and the sequence $\left\lbrace \left\lvert F_{A^{\nu}}\right\rvert^{\frac{n}{2}} \right\rbrace_{\nu \geq 1}$ is equiintegrable in $M^{n}$, i.e. for every $\varepsilon> 0,$ there exists a $\delta > 0$ such that for any measurable subset $E \subset M^{n}$ with $\left\lvert E \right\rvert < \delta,$ we have 
\begin{align}
\int_{E} \left\lvert F_{A^{\nu}}\right\rvert^{\frac{n}{2}} \leq \varepsilon \qquad \text{ for every } \nu \geq 1. 
\end{align}
Then there exists a subsequence $\left\lbrace A^{\nu_{s}}\right\rbrace_{s \geq 1}$  and an integer $s_{0}$ such that for every $s_{1}, s_{2} \geq s_{0},$ we have 
$$\left[ P^{\nu_{s_{1}}}_{A^{\nu_{s_{1}}}}\right]_{W^{k,p}} = \left[ P^{\nu_{s_{2}}}_{A^{\nu_{s_{2}}}}\right]_{W^{k,p}}. $$
Moreover, there is a bundle $P^{\infty} = \left( \left\lbrace U^{\infty}_{i}\right\rbrace_{i \in I}, \left\lbrace g_{ij}^{\infty}\right\rbrace_{i,j \in I} \right) \in \mathcal{P}_{G}^{k,p}\cap \mathcal{P}_{G}^{0}\left(M^{n}\right)$ and a limit connection $A^{\infty}\in \mathcal{U}^{k,p}\left( P^{\infty}\right) $ such that 
$$\lim_{s \rightarrow 0}\left[ P^{\nu_{s}}_{A^{\nu_{s}}}\right]_{W^{k,p}} = \left[ P^{\infty}_{A^{\infty}}\right]_{W^{k,p}}. $$
Furthermore, for every $ i \in I$,
\begin{align*}
\left( A_{Coulomb}^{\nu_{s}}\right)_{i} &\rightharpoonup A_{i}^{\infty} &&\text{ weakly in } W^{1,\frac{n}{2}}\left(U^{\infty}_{i}; \Lambda^{1}T^{\ast}U^{\infty}_{i}\otimes\mathfrak{g}\right), \\
F_{A^{\nu_{s}}_{i}} &\rightharpoonup F_{A^{\infty}_{i}} &&\text{ weakly in } L^{\frac{n}{2}}\left(U^{\infty}_{i}; \Lambda^{2}T^{\ast}U^{\infty}_{i}\otimes\mathfrak{g}\right)
\end{align*} 
and for every $i,j \in I \text{ with } U^{\infty}_{i}\cap U^{\infty}_{j} \neq \emptyset,$
\begin{align*}
g_{ij}^{\nu_{s}} \rightarrow g_{ij}^{\infty} \qquad \text{ strongly in } W^{2,q}\cap C^{0}\left(U^{\infty}_{i}\cap U^{\infty}_{j}; G\right), 
\end{align*}
for some $\frac{n}{2} < q < n,$ where $g_{ij}^{\nu_{s}}$ are the transition maps for $P_{A^{\nu_{s}}_{Coulomb}}.$
\end{theorem}
\begin{remark}
	By Dunford-Pettis theorem, the hypothesis on curvatures are of course equivalent to the sequence $\left\lbrace \left\lvert F_{A^{\nu}}\right\rvert^{\frac{n}{2}} \right\rbrace_{\nu \geq 1}$ being weakly precompact in $L^{1}.$  
\end{remark}
\begin{remark}\label{smoothcasetopologyremark}
	(i) Note that even if $P^{\nu}=P$ for all $\nu \geq 1,$ where $P \in \mathcal{P}_{G}^{k,p}\left(M^{n}\right) \cap \mathcal{P}_{G}^{0}\left(M^{n}\right),$ still $\left[P^{\infty} \right]_{C^{0}}$ may not be the same as $\left[P \right]_{C^{0}}.$  Since it is perfectly possible that $\left[P \right]_{C^{0}} \neq \left[ P_{A^{\nu}_{Coulomb}} \right]_{C^{0}}$ for infinitely many $\nu \geq 1.$ Thus in this generality, the only conclusion that we can reasonably expect is that up to a subsequence, $\left[ P_{A^{\nu}_{Coulomb}} \right]_{C^{0}}$ stabilizes to one isomorphism class and the limit is also in the same class. This is exactly what the theorem claims. \smallskip 
	
	(ii) On the other hand, by Theorem \ref{topology notions coincide}, if in addition,  $d^{\ast}A^{\nu} \in L^{\left(\frac{n}{2}, 1\right)}$ for every $\nu \geq 1,$ then  we deduce that $\left[ P_{A^{\nu}_{Coulomb}} \right]_{C^{0}}= \left[ P \right]_{C^{0}}$ for every $\nu \geq 1.$ Then the stabilization conclusion is trivial and the real content of the theorem is that $\left[ P_{\infty} \right]_{C^{0}}= \left[ P \right]_{C^{0}}$.   
\end{remark}
\begin{remark}
	Note that the condition about the existence of a common trivialization is actually an important subtle point. For an arbitrary sequence of bundles $\left\lbrace P^{\nu}\right\rbrace_{\nu \geq 1}$ over $M^{n},$ for any $x \in M^{n},$ for each $\nu \geq 1,$ we can find $U_{x}^{\nu} \subset M^{n},$ a neighborhood of $x$ such that $P^{\nu}$ is trivialized over $U_{x}^{\nu}.$ But if $\operatorname{diam}\left( U_{x}^{\nu}\right) \rightarrow 0 $ as $\nu \rightarrow \infty,$ there is no neighborhood of $x$ in $M^{n}$ over which all the bundles are trivialized simultaneously. It is not clear if this can be ruled out in general, so the condition is simply an explicit assumption to rule this situation out.    
\end{remark}
\begin{proof}
	By equiintegrability, we can find a cover $\left\lbrace V^{\infty}_{i} \right\rbrace_{i \in I}$ of $M^{n}$, which is refinement for the cover $\left\lbrace W_{\alpha}\right\rbrace_{\alpha \in J}$ and we have, for all $\nu \geq 1,$ 
	\begin{align}\label{choice}
	\left\lVert F_{A^{\nu}} \right\rVert_{L^{\frac{n}{2}}\left( V^{\infty}_{i}; \Lambda^{2}T^{\ast}V^{\infty}_{i}\otimes\mathfrak{g}\right)} \leq \min \left\lbrace \varepsilon_{Coulomb} , \frac{\varepsilon_{\Delta_{Cr}}}{4C_{Coulomb}} \right\rbrace  \quad \text{ for all } i \in I. 
	\end{align}
	Thus, $\left\lbrace V^{\infty}_{i} \right\rbrace_{i \in I}$ is a common fixed cover for the Coulomb bundles $P^{\nu}_{A^{\nu}_{Coulomb}}$ for every $\nu \geq 1.$ Denoting the transition functions of $P^{\nu}_{A^{\nu}_{Coulomb}}$ by $g^{\nu}_{ij},$ we have the Coulomb condition and the gluing relations for the Coulomb bundles, 
	\begin{align}
	&dg^{\nu}_{ij} = g^{\nu}_{ij}\left( A^{\nu}_{Coulomb}\right)_{j}  - \left( A^{\nu}_{Coulomb}\right)_{i} g^{\nu}_{ij} &&\text{ in } V^{\infty}_{i}\cap V^{\infty}_{j} \label{gluing relation nu} \\
	&d^{\ast}\left( A^{\nu}_{Coulomb}\right)_{i} = 0  \qquad &&\text{ in } V^{\infty}_{i} \label{Coulomb nu}
	\end{align} 
	for every $ i,j \in I$ with $V^{\infty}_{i}\cap V^{\infty}_{j} \neq \emptyset$ and for every $i \in I$ respectively. Also, we have the estimate 
	\begin{align}\label{common coulomb estimate}
	\left\lVert  \left( A^{\nu}_{Coulomb}\right)_{i}\right\rVert_{W^{1,\frac{n}{2}}\left( V^{\infty}_{i}; \Lambda^{2}T^{\ast}V^{\infty}_{i}\otimes\mathfrak{g}\right)}   \leq C_{Coulomb}\left\lVert F_{A^{\nu}}\right\rVert_{L^{\frac{n}{2}}\left( V^{\infty}_{i}; \Lambda^{2}T^{\ast}V^{\infty}_{i}\otimes\mathfrak{g}\right)} 
	\end{align} 
	every $i \in I.$ Combining \eqref{gluing relation nu} and \eqref{common coulomb estimate} and recalling that $G$ is compact, we deduce, 
	\begin{multline*}
	\left\lVert dg^{\nu}_{ij} \right\rVert_{L^{n}\left( V^{\infty}_{i}\cap V^{\infty}_{j}; G \right)} \\ \leq C \left( \left\lVert F_{A^{\nu}}\right\rVert_{L^{\frac{n}{2}}\left( V^{\infty}_{i}; \Lambda^{2}T^{\ast}V^{\infty}_{i}\otimes\mathfrak{g}\right)}  + \left\lVert F_{A^{\nu}}\right\rVert_{L^{\frac{n}{2}}\left( V^{\infty}_{j}; \Lambda^{2}T^{\ast}V^{\infty}_{i}\otimes\mathfrak{g}\right)} \right)
	\end{multline*}
	Since $G$ is compact, this implies $\left\lVert g^{\nu}_{ij} \right\rVert_{W^{1, n}\left( V^{\infty}_{i}\cap U^{\infty}_{j}; G \right)}$ is uniformly bounded. Thus, there exists a subsequence which converges weakly in $W^{1,n}.$ Using \eqref{common coulomb estimate} and extracting a further subsequence, we can assume that 
	\begin{align}
	g^{\nu_{s}}_{ij} &\rightharpoonup g^{\infty}_{ij} &&\text{ weakly in } W^{1,n}\left( V^{\infty}_{i}\cap V^{\infty}_{j}; G \right), \label{weakgijW1n}\\
	\left( A^{\nu_{s}}_{Coulomb}\right)_{i} &\rightharpoonup A_{i}^{\infty} &&\text{ weakly in } W^{1,\frac{n}{2}}\left(V^{\infty}_{i}; \Lambda^{1}T^{\ast}V^{\infty}_{i}\otimes\mathfrak{g}\right), \label{weakconnectionW1n2}
	\end{align}
	as $s \rightarrow 0$ for every $i,j.$ By compactness of the Sobolev embedding, up to the extraction of a further subsequence which we do not relabel, \eqref{weakgijW1n} implies 
	\begin{align}\label{strong convergence of gauges}
	g^{\nu_{s}}_{ij} \rightarrow g^{\infty}_{ij}\quad \text{ strongly in } L^{q} \text{ for every } q < \infty 
	\end{align} and since the maps $g^{\nu_{s}}_{ij}$ satisfy the cocycle conditions, passing to the limit we deduce that the maps $g^{\infty}_{ij}$ satisfy the cocycle conditions as well and thus they define a $W^{1,n}$ bundle $P^{\infty}$. Using compactness of the Sobolev embedding again, up to the extraction of a further subsequence which we do not relabel, \eqref{weakconnectionW1n2} implies 
	\begin{align}\label{strong convergence of connections}
	\left( A^{\nu_{s}}_{Coulomb}\right)_{i} \rightarrow A_{i}^{\infty}\quad \text{ strongly in } L^{r} \text{ for every } 1 \leq r < n 
	\end{align} for every $i$ and thus, we have 
	\begin{align*}
	g^{\nu_{s}}_{ij}\left( A^{\nu_{s}}_{Coulomb}\right)_{j}  - \left( A^{\nu_{s}}_{Coulomb}\right)_{i} g^{\nu_{s}}_{ij} \rightarrow  g^{\infty}_{ij}A_{i}^{\infty}  - A_{j}^{\infty} g^{\infty}_{ij} \quad \text{ in } L^{s}
	\end{align*}
	for every $s < n.$ Combining with \eqref{gluing relation nu} and \eqref{weakgijW1n}, this implies that the gluing relations 
	\begin{align}\label{gluing infty}
	dg^{\infty}_{ij} = g^{\infty}_{ij}A^{\infty}_{j}  - A^{\infty}_{i} g^{\infty}_{ij} \qquad \text{ in } V^{\infty}_{i}\cap V^{\infty}_{j}
	\end{align} 
	holds in the sense of distributions and pointwise a.e. for every $i,j$ with $V^{\infty}_{i}\cap V^{\infty}_{j} \neq \emptyset.$ Thus the local representatives $\left\lbrace A^{\infty}_{i}\right\rbrace_{i \in I}$ patch together to yield a global connection form $A^{\infty}$ on $P^{\infty}. $ Note that by \eqref{strong convergence of connections}, we also have 
	$$\left( A^{\nu_{s}}_{Coulomb}\right)_{i}\wedge \left( A^{\nu_{s}}_{Coulomb}\right)_{i} \rightarrow A_{i}^{\infty}\wedge A_{i}^{\infty}\quad \text{ strongly in } L^{\frac{r}{2}} \text{ for every } 2 \leq r < n  $$ for every $i.$ Since \eqref{weakconnectionW1n2} implies that $\left( A^{\nu_{s}}_{Coulomb}\right)_{i}\wedge \left( A^{\nu_{s}}_{Coulomb}\right)_{i}$ is uniformly bounded in $L^{\frac{n}{2}},$ by uniqueness of weak limits, we deduce 
	$$ \left( A^{\nu_{s}}_{Coulomb}\right)_{i}\wedge \left( A^{\nu_{s}}_{Coulomb}\right)_{i} \rightharpoonup A_{i}^{\infty}\wedge A_{i}^{\infty}\quad \text{ weakly in } L^{\frac{n}{2}}.$$ Combining this with \eqref{weakconnectionW1n2}, we obtain 
	$$F_{A^{\nu_{s}}_{i}} \rightharpoonup F_{A^{\infty}_{i}} \qquad \text{ weakly in } L^{\frac{n}{2}}. $$ By \eqref{Coulomb nu} and \eqref{weakconnectionW1n2}, we deduce that $A^{\infty}$ is Coulomb and thus, up to shrinking the domains which we do not rename, by Theorem \ref{regularity Coulomb}, we can assume that  $g^{\infty}_{ij}$ are $W_{loc}^{2,p}$ in $V^{\infty}_{i}\cap V^{\infty}_{j}$ for every $\frac{n}{2}< p < n$ and every $i,j$ with $V^{\infty}_{i}\cap V^{\infty}_{j} \neq \emptyset.$ Now from \eqref{gluing infty} and \eqref{gluing relation nu}, we deduce that the equation  
	\begin{multline}\label{d of diff}
	d\left( g^{\nu_{s}}_{ij} - g^{\infty}_{ij}\right) = \left( g^{\nu_{s}}_{ij} - g^{\infty}_{ij}\right)\left( A^{\nu_{s}}_{Coulomb}\right)_{j} - \left( A^{\nu_{s}}_{Coulomb}\right)_{i}\left( g^{\nu_{s}}_{ij} - g^{\infty}_{ij}\right) 
	\\ + g^{\infty}_{ij}\left[ \left( A^{\nu_{s}}_{Coulomb}\right)_{j} - A^{\infty}_{j}\right] -  \left[ \left( A^{\nu_{s}}_{Coulomb}\right)_{i} - A^{\infty}_{i}\right]g^{\infty}_{ij}
	\end{multline}
	holds in $V^{\infty}_{i}\cap V^{\infty}_{j}$ whenever the intersection is non-empty. Since $A^{\nu_{s}}_{Coulomb}$ and $A^{\infty}$ are both Coulomb, we deduce the equation
	\begin{multline}\label{laplacian of the diff}
	-\Delta u^{\nu_{s}}_{ij} = \ast \left[ d u^{\nu_{s}}_{ij} \wedge  \ast \left( A^{\nu_{s}}_{Coulomb}\right)_{j}\right] + \ast \left[  \ast \left( A^{\nu_{s}}_{Coulomb}\right)_{i} \wedge du^{\nu_{s}}_{ij}\right] \\
	 +\ast \left[ d g^{\infty}_{ij} \wedge  \ast \left[ \left( A^{\nu_{s}}_{Coulomb}\right)_{j} - A^{\infty}_{j}\right]\right] \\ + \ast \left[  \ast \left[ \left( A^{\nu_{s}}_{Coulomb}\right)_{i} - A^{\infty}_{i}\right] \wedge dg^{\infty}_{ij}\right] 
	\end{multline} 
	in $V^{\infty}_{i}\cap V^{\infty}_{j}$ whenever the intersection is non-empty, where $u^{\nu_{s}}_{ij}=g^{\nu_{s}}_{ij} - g^{\infty}_{ij}.$
	Now we choose exponents $\frac{n}{2} < p < n$ and $ 1 < r < n$ such that $\frac{1}{n} < \frac{1}{r} + \frac{n-p}{np} < \frac{2}{n}.$ Note that since $p > \frac{n}{2}$ implies $\frac{np}{n-p}> n,$ such a choice of $r$ is possible. Now, by slightly shrinking the open sets $V^{\infty}_{i},$ we can find open sets $\tilde{U}^{\infty}_{i}, \tilde{\tilde{U}}^{\infty}_{i}$ such that $\left\lbrace \tilde{U}^{\infty}_{i}\right\rbrace_{i \in I} $ is still a cover for $M^{n}$ and for each $i \in I,$ we have, 
	 $ \tilde{U}^{\infty}_{i}\subset \subset \tilde{\tilde{U}}^{\infty}_{i}\subset \subset V^{\infty}_{i}. $ Now considering the equation \eqref{laplacian of the diff} in $\tilde{\tilde{U}}^{\infty}_{i}\cap \tilde{\tilde{U}}^{\infty}_{j}$, whenever the intersection is non-empty and recalling that $g^{\infty}_{ij}$ is $W^{2,p}$ in $\tilde{\tilde{U}}^{\infty}_{i}\cap \tilde{\tilde{U}}^{\infty}_{j},$ using Lemma \ref{ellipticCritical}, by our choice in \eqref{choice}, we deduce the estimate 
	 \begin{multline}
	 \left\lVert u^{\nu_{s}}_{ij}\right\rVert_{W^{2,q}\left( \tilde{U}^{\infty}_{i}\cap \tilde{U}^{\infty}_{i}; G\right)} \leq C \left\lVert u^{\nu_{s}}_{ij}\right\rVert_{W^{1,2}\left( \tilde{\tilde{U}}^{\infty}_{i}\cap \tilde{\tilde{U}}^{\infty}_{j}\right)} \\ + C \left\lVert d^{\ast}g^{\infty}_{ij}\right\rVert_{L^{\frac{np}{n-p}}\left( \tilde{\tilde{U}}^{\infty}_{i}\cap\tilde{\tilde{U}}^{\infty}_{j}\right)} \left\lVert \left( A^{\nu_{s}}_{Coulomb}\right)_{i} - A^{\infty}_{i} \right\rVert_{L^{r}\left( \tilde{\tilde{U}}^{\infty}_{i}\right)} \\ + C \left\lVert d^{\ast}g^{\infty}_{ij}\right\rVert_{L^{\frac{np}{n-p}}\left( \tilde{\tilde{U}}^{\infty}_{i}\cap\tilde{\tilde{U}}^{\infty}_{j}\right)}\left\lVert \left( A^{\nu_{s}}_{Coulomb}\right)_{j} - A^{\infty}_{j} \right\rVert_{L^{r}\left( \tilde{\tilde{U}}^{\infty}_{j}\right)},
	 \end{multline}
	 where $n > q = \frac{npr}{np + r(n-p)} > \frac{n}{2}.$
	 Now \eqref{strong convergence of connections} implies that the last two terms on the right hand side of the estimate above converges to zero as $s \rightarrow \infty.$ The first term also converges to zero by  \eqref{strong convergence of gauges},\eqref{strong convergence of connections} and \eqref{d of diff}. Thus, by Sobolev embedding, we obtain 
	 \begin{align*}
	 \left\lVert g^{\nu_{s}}_{ij} - g^{\infty}_{ij}\right\rVert_{C^{0}\left( \overline{ \tilde{U}^{\infty}_{i}\cap \tilde{U}^{\infty}_{j}}; G\right)} \rightarrow 0 \text{ as } s \rightarrow \infty
	 \end{align*} for each $i,j \in I$ with $\tilde{U}^{\infty}_{i}\cap \tilde{U}^{\infty}_{j}\neq \emptyset.$ By Corollary 3.3 in \cite{UhlenbeckGaugefixing}, we deduce the existence of
	 a smaller cover $\left\lbrace U^{\infty}_{i}\right\rbrace_{i \in I}$ and gauge changes $\sigma_{i} \in W^{2,q}\left( U^{\infty}_{i}; G \right)$ satisfying $$ g^{\infty}_{ij}= \sigma_{i}^{-1} g^{\nu_{s_{0}}}_{ij} \sigma_{j} \qquad \text{ in } U^{\infty}_{i}\cap U^{\infty}_{j},$$
	 whenever the intersection is non-empty, for some integer $s_{0}$ large enough. This proves the result.  
\end{proof}
As an immediate consequence, we obtain Theorem \ref{improved strong convergence}. 
\begin{proof}\textbf{( of Theorem \ref{improved strong convergence} )}
	In light of the discussions in Remark \ref{smoothcasetopologyremark} (ii), this is just a restatement of Theorem \ref{equiintegrability theorem} in this case. 
\end{proof}  
We now prove that $\frac{n}{2}$-Yang-Mills energy can detect \emph{topological flatness} in a $W^{k,p}$ principal $G$-bundle equipped with a $\mathcal{U}^{k,p}$ connection for $kp=n$. 
\begin{theorem}[flatness criterion]\label{flatness in W1n bundles}
Let $kp=n.$ For any cover $\mathcal{U}$ of $M^{n},$ there exists a constant $\delta >0,$ depending only on $\mathcal{U}$, $M^{n}$ and $G$ such that if $ P $ is a $W^{k,p}$ bundle trivialized over $\mathcal{U}$ and $A$ is a $\mathcal{U}^{k,p}$ connection on $P,$ then either $YM_{n/2}\left(A \right)  > \delta , $  or $\left[ P_{A}\right]_{W^{1,n}} = \left[P^{0}\right]_{C^{0}}$, where $P^{0}$ is a flat bundle.
\end{theorem}
\begin{proof}
	Since for $kp=n,$ every $W^{k,p}$ bundle is also a $W^{1,n}$ bundle and every $\mathcal{U}^{k,p}$ connection is also a $\mathcal{U}^{1,n}$ connection, it is enough to prove for $k=1.$ If the result is false, there for every $\nu \geq 1,$ there exists a $W^{1,n}$ bundle $P^{\nu}$ trivialized over $\mathcal{U}$ with a $\mathcal{U}^{1,n}$ connection $A^{\nu} \in \mathcal{U}^{1,n}\left(P^{\nu}\right)$ such that $\left[ P^{\nu}_{A^{\nu}}\right]_{W^{1,n}} \neq \left[P^{0}\right]_{C^{0}}$ for any flat bundle $P^{0}$ and $F_{A^{\nu}} \rightarrow 0$ in $L^{\frac{n}{2}}.$ Since the strong convergence of the curvatures in $L^{\frac{n}{2}}$ implies the equiintegrability, applying Theorem \ref{equiintegrability theorem}, we deduce that $\left[ P^{\nu_{s}}_{A^{\nu_{s}}}\right]_{W^{1,n}} = \left[ P^{\infty}_{A^{\infty}}\right]_{W^{1,n}},$ for $s$ large enough. But it is easy to see that by the strong convergence of the curvatures in $L^{\frac{n}{2}}$ to zero, we have $dg_{ij}^{\infty} = 0$ in $U^{\infty}_{i}\cap U^{\infty}_{j},$ for every $i,j \in I$ with $U^{\infty}_{i}\cap U^{\infty}_{j} \neq \emptyset.$ This means $P^{\infty}$ is a flat bundle. This contradiction proves the theorem.    
\end{proof}
\noindent Combining Theorem \ref{flatness in W1n bundles} with Theorem \ref{topology notions coincide}, we immediately obtain Theorem \ref{energy gap}. 
\begin{appendices}
	\section{$G$-valued Sobolev maps}\label{Gvaluedmaps}  Without loss of generality, we can always assume that the compact finite dimensional Lie group $G$ is endowed with a bi-invariant metric and is smoothly embedded isometrically in $\mathbb{R}^{N_{0}}$ for some, possibly quite large, integer $N_{0} \geq 1.$ By compactness of $G,$ there exists a constant $C_{G} > 0$ such that $G \subset \subset B_{C_{G}}(0) \subset \mathbb{R}^{N_{0}}.$ 
		
		\begin{definition}
			Let $U \subset \mathbb{R}^{n}$ be open. The  space $W^{k,p}\left( U ; G \right)$ is defined as 
			$$W^{k,p}\left( U ; G \right) : = \left\lbrace  f \in W^{k,p}\left( U ; \mathbb{R}^{N_{0}}\right) : f(x) \in G \text{ for a.e. } x \in U  \right\rbrace.  $$
		\end{definition}
		The compactness of $G$ implies that $W^{k,p}\left( U ; G \right) \subset L^{\infty}\left( U ; \mathbb{R}^{N_{0}} \right)$  with the bounds $\left\lVert f \right\rVert_{L^{\infty}\left( U ; \mathbb{R}^{N_{0}} \right)} \leq C_{G}$ for any $f \in W^{k,p}\left( U ; G \right).$ By the Gagliardo-Nirenberg inequality, it follows that $W^{k,p}\left( U ; G \right)$ is an infinite dimensional topological group with respect to the topology in inherits as a topological subspace of the  Banach space  $W^{k,p}\left( U ; \mathbb{R}^{N_{0}}\right)$. Note that $W^{k,p}\left( U ; G \right)$ is not even a linear space, so there is no question of a norm. It inherits only a topology from the norm topology of $W^{k,p}\left( U ; \mathbb{R}^{N_{0}}\right).$ 
		
		On the other hand, since the Lie algebra of $G$, i.e. $\mathfrak{g}$ is a linear space and consequently so is $\Lambda^{k}\mathbb{R}^{n}\otimes\mathfrak{g}$ for any $0 \leq k \leq n,$ the space of $\mathfrak{g}$-valued $k$-forms of class $W^{k,p}$ is defined by requiring each scalar component of the maps to be $W^{k,p}$ functions in the usual sense. The standard properties of Sobolev functions, including smooth approximation by mollification, carry over immediately to this setting by arguing componentwise. The stark contrast between the two settings is due to the fact that in general a map $ g \in W^{k,p}\left( U; G\right)$ need not have a $W^{k,p}$ `lift' to the Lie algebra. More precisely, there need not exist a map $ u \in W^{k,p}\left( U; \mathfrak{g}\right)$ with the property that $g = \operatorname{exp}\left( u \right),$ where $\operatorname{exp}: \mathfrak{g} \rightarrow G$ is the exponential map of $G.$ However, $\mathfrak{g} = T_{\mathbf{1}_{G}}G$ and there exists a small enough $C^{0}$-neighborhood of the identity element $\mathbf{1}_{G} \in G$ in $G$ such the exponential map is a local smooth diffeomorphism onto that neighborhood. We shall use this fact crucially and repeatedly, so we fix some notations.  
		\begin{notation}
			Let $G$ be a compact finite dimensional Lie group and $\mathcal{O}_{G} \subset G$ be a neighborhood of the identity in $G$ which is contained in the domain of the inverse of the exponential map, i.e. $ \mathbf{1}_{G} \in \mathcal{O}_{G} \subset  \operatorname{Dom}(\operatorname{exp}^{-1}),$ such that $\operatorname{exp}^{-1}\left( \mathcal{O}_{G}\right) \subset \mathfrak{g}$ is a convex set containing the origin, i.e. $0 \in \operatorname{exp}^{-1}\left( \mathcal{O}_{G}\right) \subset \mathfrak{g}$ is convex.
		\end{notation} 
	Now we prove a few lemmas for $kp >n,$ which are $W^{k,p}$-analogues of classical results about $G$-valued continuous maps. All subsets of $\mathbb{R}^{n}$ are always assumed to be at least Lipschitz sets. 
	\begin{lemma}\label{patching}
		Let $d_{0}>0$ be a real number and $kp > n.$ Let $U \subset \mathbb{R}^{n}$ be open, bounded, convex and let $A, B \subset U$ be two closed convex subsets such that $B \subset \subset A$ and $\operatorname{dist}\left( B; \partial A \right)> d_{0}.$ Then given any \emph{continuous} map $F \in W^{k,p} \left( A ; \mathcal{O}_{G}\right) ,$
		there exists a \emph{continuous} map $\tilde{F}\in W^{k,p} \left( U ; \mathcal{O}_{G}\right)$ such that $\tilde{F} = F$ on a neighborhood of $B$ and $\tilde{F} = \mathbf{1}_{G}$ on a neighborhood of $U\setminus \operatorname{int}\left( A\right).$ Moreover, there exists a constant $C_{1} = C_{1}\left( d_{0}, n, k, p, G \right) \geq 1,$ nonincreasing with $d_{0},$ such that we have the estimate
		\begin{equation*}
		\left\lVert \operatorname{exp}^{-1} \left( \tilde{F}\right)  \right\rVert_{W^{k,p}\left( A; \mathfrak{g}\right)} \leq C_{1}\left\lVert \operatorname{exp}^{-1} \left( F\right) \right\rVert_{W^{k,p}\left( A; \mathfrak{g}\right)}.
		\end{equation*}
	\end{lemma}	
	\begin{proof}
		Since $B$ and $U\setminus \operatorname{int}\left( A\right)$ are disjoint, one can construct a smooth map $\psi: U \rightarrow \left[0,1\right]$ such that $\psi \equiv 0$ in a neighborhood of $U\setminus \operatorname{int}\left( A\right)$ and $\psi \equiv 1$ in a neighborhood of $B.$ Then we set $$ F_{\psi}(x):= \operatorname{exp}\left( \psi(x)\operatorname{exp}^{-1}\left[ F(x)\right]\right) \qquad \text{ for all } x \in A.$$ Clearly, $F_{\psi}$ takes values in $\mathcal{O}_{G}$ by convexity of $\operatorname{exp}^{-1}\left( \mathcal{O}_{G}\right).$ But since $F_{\psi} \equiv \mathbf{1}_{G}$ near the boundary of $A,$ the map 
		\begin{align*}
		\tilde{F}(x):= \left\lbrace \begin{aligned}
		&F_{\psi}(x) &&\text{ if } x \in A, \\
		&\mathbf{1}_{G} &&\text{ if } x \in U\setminus \operatorname{int}\left( A\right),
		\end{aligned}\right. 
		\end{align*} 
		is continuous and satisfies all our requirements. By the smoothness of $\psi$ and the exponential map, the Sobolev bounds follow from straight forward computation and obvious estimates. The only dependence of the constant $C_{1}$ on $d_{0}$ is via the $L^{\infty}$ norms of the derivatives of $\psi$ and hence is nonincreasing.  
	\end{proof} 
As a consequence, we deduce 
\begin{lemma}[Extension]\label{contextension}
	Let $d_{0}>0$ be a real number and $kp > n.$ Let $U, V, W \subset \mathbb{R}^{n}$ be convex open sets such that $W \subset \subset V \subset \subset U,$ $\operatorname{dist}\left( W; \partial V \right), \operatorname{dist}\left( V; \partial U \right)> d_{0}$ and  $U$ is bounded. Then there exists a constant $\delta_{G} = \delta_{G}(n, G) > 0$ such that for any two maps $f \in W^{k,p}\left(U;G\right)$ and $g \in C^{\infty}\left(V; G\right)$ satisfying the bound $$ \left\lVert f^{-1}g - \mathbf{1}_{G}\right\rVert_{W^{k,p}\left( V; G\right)} \leq \delta_{G},$$ we can find a map $\tilde{f} \in W^{k,p}\left(U;G\right)$ such that $\tilde{f}=g$ in a neighborhood of $W$ and $\tilde{f}= f$ in a neighborhood of $U\setminus V$ in $U.$ Moreover, there exists a constant $C_{2} = C_{2}\left( d_{0}, n, k, p, G \right) \geq 1,$ nonincreasing with $d_{0},$ such that we have the estimate
	\begin{equation*}
	\left\lVert \operatorname{exp}^{-1} \left( f^{-1}\tilde{f}\right)  \right\rVert_{W^{k,p}\left( V; \mathfrak{g}\right)} \leq C_{2}\left\lVert \operatorname{exp}^{-1} \left( f^{-1}g\right) \right\rVert_{W^{k,p}\left( V; \mathfrak{g}\right)}.
	\end{equation*}
\end{lemma}
\begin{proof}
	We use the previous lemma by choosing $B= \overline{W},$ $A= \overline{V}$ and $F= f^{-1}g.$ Clearly, there is a smallness parameter $\delta$ as claimed such that the bound forces $f^{-1}g$ to take values in a neighborhood $\mathcal{O}_{G} \subset G$ which satisfies the assumptions of Lemma \ref{patching}. Now we set $$\tilde{f}(x) := f(x)\tilde{F}(x) \qquad \text{ for all } x \in U,$$ where $\tilde{F}$ is the extension of $F=f^{-1}g$ given by Lemma \ref{patching}. The estimate follows by simple computations and obvious estimates.
\end{proof} 
\begin{lemma}[Relative smooth approximation]\label{smoothinglemma}
	Let $kp > n.$ Let $U, V, W \subset \mathbb{R}^{n}$ be convex open sets such that $W \subset \subset V \subset \subset U$ and $U$ is bounded. Then there exists a constant $\delta_{G} = \delta_{G} (n, G ) > 0$  such that if $f \in W^{k,p} \left( U; G\right) $ satisfies $$ \operatorname{osc}\left( f; V \right) \leq \delta_{G},$$ then for every $\varepsilon > 0,$ there exists a smooth map $f^{\varepsilon} \in  C^{\infty} \left( W ; G\right) $ such that $$ \left\lVert f^{-1}f^{\varepsilon} - \mathbf{1}_{G}\right\rVert_{W^{k,p}\left( W; G\right)} \leq \varepsilon.$$ Moreover, if $A \subset U$ is a closed subset ( possibly empty ) such that $f$ is smooth in a neighborhood of $A \cap \overline{W}$, then $f^{\varepsilon}$ can be chosen to ensure $f^{\varepsilon} = f$ in $A \cap \overline{W}.$
	
\end{lemma}
\begin{proof}
	If $f$ takes values in a finite dimensional vector space, the lemma is completely classical ( see e.g. Theorem 2.5 in \cite{Hirsch_differentialtopology} ) without requiring any smallness condition on the oscillations of $f.$ So we just need to choose $\delta_{G}$ small enough such that $\left( \bar{f}\right)^{-1}f \subset \mathcal{O}_{G},$ for some constant element $\bar{f}\in G ,$ where $\mathcal{O}_{G} \subset G$ is as in Lemma \ref{patching}. Then we can write $$f(x) = \bar{f}\operatorname{exp} u(x) \qquad \text{ for all }  x \in V$$ for some continuous map $u \in W^{k,p}\left( V ; \mathfrak{g}\right) .$ Note that since $f$ is smooth in a neighborhood of $A \cap \overline{W},$ which we can assume to be contained inside $V$ and we have $$ u(x) = \operatorname{exp}^{-1}\left[ \left(\bar{f}\right)^{-1}f(x) \right] \qquad \text{ for all }  x \in V,$$ we infer that $u$ is smooth in a neighborhood of $A \cap \overline{W}.$ Since $\mathfrak{g}$ is a finite dimensional vector space, we can find a sequence $\left\lbrace u^{\varepsilon} \right\rbrace \subset C^{\infty}\left( \overline{W} ; \mathfrak{g}\right) $ such that $$u^{\varepsilon} \stackrel{W^{k,p}}{\rightarrow} u \quad \text{ in } \overline{W} \text{ as } \varepsilon \rightarrow 0 \qquad \text{ and } \qquad  u^{\varepsilon}\big|_{A \cap \overline{W}} = u\big|_{A \cap \overline{W}}.$$ Choosing $\varepsilon > 0$ small enough and reducing $\delta$ if necessary, we can ensure that the image of $u^{\varepsilon}$ is contained in $\operatorname{exp}^{-1}\left( \mathcal{O}_{G}\right)$ as well. Now $f^{\varepsilon}(x) := \bar{f}\operatorname{exp} \left[u^{\varepsilon}(x)\right]$ is the desired map, which satisfies the estimate if we choose $\varepsilon$ suitably small.   
\end{proof}   
\section{Smooth Approximation in subcritical regime}\label{proof of subcritical approx}
Now we prove the smooth approximation theorem for $W^{k,p}$ bundles for $kp > n.$ 

\begin{proof}\textbf{( of Theorem \ref{smoothingC0bundle} )}	We prove only the case $k=1.$ The case $k = 2$ is similar. Also, we only show the existence of an approximating smooth cocycle $h_{ij}.$ The existence of the maps $\sigma_{i}$ follows, as already proved for $k=2$ by Uhlenbeck in \cite{UhlenbeckGaugefixing}, Corollary 3.3, but the argument works for $k=1$ as well.\smallskip  
	
	\noindent During the course of this proof, we will freely reduce $\varepsilon > 0$ finitely many times in order to make it suitably small. Also, we set $\delta_{G}> 0$ to be the smaller of the two smallness parameters $\delta_{G}$ given by Lemma \ref{contextension} and Lemma \ref{smoothinglemma}. All opens sets we are going to chose below are always assumed to be at least Lipschitz, convex and bounded without further comment. Now the rest of the proof involves two nested induction arguments.\smallskip 
	
	\noindent \textbf{Step 1: Outer induction::} We begin by choosing $N^{2}$ open sets $ \left\lbrace V^{r}_{i} \right\rbrace_{ 1 \leq  i,r \leq N}$ such that we have $M^{n}= \bigcup\limits_{i=1}^{N}V^{N}_{i}$ and for each $ 1 \leq i\leq N$ we have the inclusions $$ V^{N}_{i}\subset \subset V^{N-1}_{i} \subset \subset \ldots \subset \subset V_{i}^{r+1} \subset \subset V^{r}_{i}\subset \subset \ldots \subset \subset V^{1}_{i} \subset \subset V_{i} \subset \subset U_{\phi(i)},$$ together with the smallness conditions \begin{equation}\label{small_osc}
	\operatorname{osc}\left( g_{ij}; \overline{V_{i}}\cap \overline{V_{j}}\right) \leq \frac{\delta_{G}}{2} .
	\end{equation} Here and henceforth $g_{ij}$ stands for $g_{\phi(i)\phi(j)},$ where $\phi$ is the refinement map. Now, once we have chosen our sets $ \left\lbrace V^{r}_{i} \right\rbrace_{ 1 \leq  i,r \leq N},$ there exists a number $d_{0}>0$ such that we have $$ \operatorname{dist}\left( V^{r+1}_{i}; \partial V^{r}_{i} \right) \geq \left( N+1\right)^{4}d_{0} \qquad \text{ for all } 1 \leq r \leq N.$$ Let $C_{2}(d_{0}) \geq 1$ be the constant given by Lemma \ref{contextension} for this $d_{0}>0.$ Now we set $$ C_{0}:= 100^{nN}\left[ C_{\operatorname{exp}}C_{G}C_{2}(d)\right]^{4},$$ 
	where $C_{G} \geq 1$ is an $L^{\infty}$ bound for $G$ and $C_{\operatorname{exp}} \geq 1$ is a $C^{2}$ bound for the smooth maps $\operatorname{exp}$ and $\operatorname{exp}^{-1}$ for $G.$ \smallskip 
	
	\noindent \textbf{Step 1a: Hypotheses for outer induction:: } For each $1 \leq r \leq N,$ we want to inductively construct a collection of smooth maps $\left\lbrace g^{r}_{ij}\right\rbrace_{1 \leq i,j\leq r} $ such that $g^{r}_{ij}: V^{r}_{i}\cap V^{r}_{j} \rightarrow G$ is smooth, $g^{ii}_{r}=\mathbf{1}_{G}$ for all $1 \leq i \leq r$  and satisfies
	\begin{itemize}
		\item[(H1)]  the cocycle condition \begin{equation}\label{cocycle g}
		g^{r}_{ij}g^{r}_{jk} = g^{r}_{ik} \quad \text{ for all } x \in V^{r}_{i}\cap V^{r}_{j}\cap V^{r}_{k},  \text{ for all } 1 \leq i,j,k \leq r, \tag{$\mathbf{H}^{1}_{r}$}
		\end{equation}
		\item[(H2)] and the estimate 
		\begin{equation}\label{estimate g}
		\left\lVert g^{r}_{ij} - g_{ij}\right\rVert_{W^{k,p}\cap C^{0}\left( V^{r}_{i}\cap V^{r}_{j};G\right)}  \leq \left(2C_{0}\right)^{-\frac{N^{2}}{r^{2}}}\varepsilon \tag{$\mathbf{H}^{2}_{r}$}.
		\end{equation}
	\end{itemize}
	Setting $g^{1}_{11}= \mathbf{1}_{G}$ in $V_{1}^{1},$ the hypotheses are met for $r=1.$ So  we assume that we have already constructed such a family for all  $1\leq r \leq r_{0}$ for some $1 \leq r_{0} \leq N-1$ and show that we can construct such a family for $r = r_{0}+1.$
	\smallskip 
	
	\noindent \textbf{Step 1b: Induction step for outer induction:: }
	Given such a family of smooth maps	$\left\lbrace g^{r_{0}}_{ij}\right\rbrace_{1 \leq i,j\leq r_{0}} $ where $g^{r_{0}}_{ij}: V^{r_{0}}_{i}\cap V^{r_{0}}_{j} \rightarrow G$ for all $1\leq i, j\leq r_{0}$ that satisfies \eqref{cocycle g} and \eqref{estimate g} for $r=r_{0},$ we can define the maps $g^{r_{0}+1}_{ij}: V^{r_{0}+1}_{i}\cap V^{r_{0}+1}_{j} \rightarrow G$ as the restriction $$ g^{r_{0}+1}_{ij}:= g^{r_{0}}_{ij}\Big|_{V^{r_{0}+1}_{i}\cap V^{r_{0}+1}_{j}}.$$ Also, we obviously set $g^{r_{0}+1}_{(r_{0}+1)(r_{0}+1)} = \mathbf{1}_{G}$ in 
	$V^{r_{0}+1}_{r_{0}+1}.$ Thus, it only remains to construct smooth maps $g^{r_{0}+1}_{i(r_{0}+1)}: V^{r_{0}+1}_{i}\cap V^{r_{0}+1}_{r_{0}+1} \rightarrow G$ for the values $1\leq i \leq r_{0},$ such that they satisfy, for all $1 \leq i < j \leq r_{0}, $ the identity 
	\begin{equation}\label{cocycle r0+1}
	g^{r_{0}+1}_{j(r_{0}+1)}(x) = \left[ g^{r_{0}+1}_{ij}(x)\right]^{-1}g^{r_{0}+1}_{i(r_{0}+1)}(x), 
	\end{equation}  for all $ x \in V^{r_{0}+1}_{i}\cap V^{r_{0}+1}_{j}\cap V^{r_{0}+1}_{r_{0}+1},$ along with the estimates \begin{equation}
	\left\lVert g^{r_{0}+1}_{i(r_{0}+1)} - g_{i(r_{0}+1)}\right\rVert_{W^{k,p}\cap C^{0}\left( V^{r_{0}+1}_{i}\cap V^{r_{0}+1}_{r_{0}+1};G\right)}  \leq \left(2C_{0}\right)^{-\frac{N^{2}}{(r_{0}+1)^{2}}}\varepsilon ,   
	\end{equation}
	for all $ 1\leq i \leq r_{0}.$ This will be done by another induction.\smallskip 
	
	\noindent \textbf{Step 2: Inner induction:: } To do this, we first chose $\left(r_{0}+1\right)^{2}$ open sets $\left\lbrace W^{l}_{i}\right\rbrace_{1 \leq i,l \leq r_{0}+1}$ such that for each $ 1 \leq i\leq r_{0}+1$ we have the inclusions $$ V^{r_{0}+1}_{i}\subset \subset W^{r_{0}+1}_{i} \subset \subset \ldots \subset \subset W_{i}^{l+1} \subset \subset W^{l}_{i}\subset \subset \ldots \subset \subset W^{1}_{i} \subset \subset V^{r_{0}}_{i}$$ and for all $ 1\leq i \leq r_{0}+1, 1 \leq l \leq r_{0} ,$ we have, $$\operatorname{dist}\left( W^{l+1}_{i};  \partial W^{l}_{i} \right), \operatorname{dist}\left( V^{r_{0}+1}_{i}, \partial W^{r_{0}+1}_{i}\right), \operatorname{dist}\left( W^{1}_{i}, \partial V^{r_{0}}_{i}\right)  \geq d_{0} .$$  \smallskip 
	
	\noindent \textbf{Step 2a: Hypotheses for inner induction::} Let $1 \leq l_{0} \leq r_{0}-1$ and suppose we have constructed maps $h^{l_{0}}_{i(r_{0}+1)} \in C^{\infty}\left( W^{l_{0}}_{i}\cap W^{l_{0}}_{r_{0}+1}; G\right)$  for all $1 \leq i \leq l_{0}$ satisfying 
	\begin{equation}\label{cocycle_h}
	h^{l_{0}}_{i(r_{0}+1)}\left( h^{l_{0}}_{j(r_{0}+1)}\right)^{-1} = g^{r_{0}}_{ij}\qquad \text{ in } 
	W^{l_{0}}_{i}\cap W^{l_{0}}_{j}\cap W^{l_{0}}_{r_{0}+1}, 
	\end{equation}
	for all $1 \leq i,j \leq l_{0},$ together with the estimates 
	\begin{equation}\label{estimate_h}
	\left\lVert h^{l_{0}}_{i(r_{0}+1)} - g_{i(r_{0}+1)}\right\rVert_{W^{k,p}\cap C^{0}\left( W^{l_{0}}_{i}\cap W^{l_{0}}_{r_{0}+1};G\right)}  \leq \left(2C_{0}\right)^{-\frac{N^{2}}{l_{0}(r_{0}+1)}}\varepsilon  \text{ for all } 1\leq i \leq l_{0}.   
	\end{equation}
	Now we construct maps $h^{l_{0}+1}_{i(r_{0}+1)}: W^{l_{0}+1}_{i}\cap W^{l_{0}+1}_{r_{0}+1} \rightarrow G $  for all $1 \leq i \leq l_{0}+1$ satisfying 
	\begin{equation}\label{cocycle_h+1}
	h^{l_{0}+1}_{i(r_{0}+1)}\left( h^{l_{0}+1}_{j(r_{0}+1)}\right)^{-1} = g^{r_{0}}_{ij}\qquad \text{ in } 
	W^{l_{0}+1}_{i}\cap W^{l_{0}+1}_{j}\cap W^{l_{0}+1}_{r_{0}+1}, 
	\end{equation}
	for all $1 \leq i,j \leq l_{0}+1,$ together with the estimates 
	\begin{equation}\label{estimate_h+1}
	\left\lVert h^{l_{0}+1}_{i(r_{0}+1)} - g_{i(r_{0}+1)}\right\rVert_{W^{k,p}\cap C^{0}\left( W^{l_{0}+1}_{i}\cap W^{l_{0}+1}_{r_{0}+1};G\right)}  \leq \left(2C_{0}\right)^{-\frac{N^{2}}{(l_{0}+1)(r_{0}+1)}}\varepsilon  ,   
	\end{equation}
	for all $ 1\leq i \leq l_{0}+1.$ Note that by virtue of \eqref{small_osc}, we can use Lemma \ref{smoothinglemma} with $A = \emptyset$ to construct a smooth map $h^{1}_{1(r_{0}+1)} \in C^{\infty}\left( W^{1}_{i}\cap W^{1}_{r_{0}+1}; G \right)$ such that we have 
	$$\left\lVert h^{1}_{1(r_{0}+1)} - g_{1(r_{0}+1)}\right\rVert_{W^{k,p}\cap C^{0}\left( W^{1}_{i}\cap W^{1}_{r_{0}+1};G\right)}  \leq \left(2C_{0}\right)^{-\frac{N^{2}}{(r_{0}+1)}}\varepsilon .  $$ 
	This implies that \eqref{cocycle_h} and \eqref{estimate_h} are satisfied for $l_{0}=1$ and we can start the induction. \smallskip 
	
	\noindent \textbf{Step 2b: Induction step for inner induction:: } As before, by restricting already constructed maps, it only remains to construct \emph{one} smooth map $h^{l_{0}+1}_{(l_{0}+1)(r_{0}+1)} \in C^{\infty}\left( W^{l_{0}+1}_{l_{0}+1}\cap W^{l_{0}+1}_{r_{0}+1} ; G\right)$ such that 
	\begin{equation}\label{cocycle_l+1}
	h^{l_{0}+1}_{(l_{0}+1)(r_{0}+1)} = \left[ g^{r_{0}}_{i(l_{0}+1)}\right]^{-1}h^{l_{0}}_{i(r_{0}+1)}
	\qquad \text{ in }   W^{l_{0}+1}_{i}\cap W^{l_{0}+1}_{l_{0}+1}\cap W^{l_{0}+1}_{r_{0}+1}
	\end{equation}
	for all $1 \leq i \leq l_{0}$ and we have the estimate 
	\begin{equation}\label{estimate_l+1}
	\left\lVert h^{l_{0}+1}_{(l_{0}+1)(r_{0}+1)} - g_{(l_{0}+1)(r_{0}+1)}\right\rVert_{W^{k,p}\cap C^{0}\left( W^{l_{0}+1}_{l_{0}+1}\cap W^{l_{0}+1}_{r_{0}+1};G\right)}  \leq \left(2C_{0}\right)^{-\frac{N^{2}}{(l_{0}+1)(r_{0}+1)}}\varepsilon   .   
	\end{equation}
	 We define $\hat{h}_{(l_{0}+1)(r_{0}+1)}: \bigcup\limits_{1 \leq i \leq l_{0}} \left( W^{l_{0}}_{i}\cap W_{l_{0}+1}^{l_{0}}\cap W^{l_{0}}_{r_{0}+1} \right) \rightarrow G$ by setting 
	$$\hat{h}_{(l_{0}+1)(r_{0}+1)}:=  \left[ g^{r_{0}}_{i(l_{0}+1)}\right]^{-1}h^{l_{0}}_{i(r_{0}+1)} \qquad \text{ in } W^{l_{0}}_{i}\cap W^{l_{0}+1}_{l_{0}}\cap W^{l_{0}}_{r_{0}+1}$$ for all $1 \leq i \leq l_{0}.$ Note that $g^{r_{0}}_{i(l_{0}+1)}$ and $h^{l_{0}}_{i(r_{0}+1)}$ are already defined  in $V_{i}^{r_{0}}\cap V_{l_{0}+1}^{r_{0}}$ and $W^{l_{0}}_{i}\cap W^{l_{0}}_{r_{0}+1}$ respectively, by the induction hypotheses.
	By \eqref{cocycle g} for $r=r_{0}$ and \eqref{cocycle_h}, the definitions agree in quadruple intersections and thus
	$\hat{h}_{(l_{0}+1)(r_{0}+1)}: \bigcup\limits_{1 \leq i \leq l_{0}} \left( W^{l_{0}}_{i}\cap W_{l_{0}+1}^{l_{0}}\cap W^{l_{0}}_{r_{0}+1} \right) \rightarrow G $ is actually smooth. 
	By using \eqref{estimate_h}, \eqref{estimate g} for $r=r_{0}$ and noting that $l_{0}+1 \leq r_{0},$  one can estimate the norm $$\left\lVert  g_{(l_{0}+1)(r_{0}+1)}^{-1}\hat{h}_{(l_{0}+1)(r_{0}+1)} - \mathbf{1}_{G}\right\rVert_{W^{k,p}\left(\bigcup\limits_{1 \leq i \leq l_{0}} \left( W^{l_{0}}_{i}\cap W_{l_{0}+1}^{l_{0}}\cap W^{l_{0}}_{r_{0}+1} \right) ; G \right)} .$$
	Now, choosing $\varepsilon > 0$ suitably small, this norm can be made smaller than $\delta_{G}.$ Pick an open set $X$ satisfying  $$ \bigcup\limits_{1 \leq i \leq l_{0}} \left( W^{l_{0}+1}_{i}\cap W_{l_{0}+1}^{l_{0}+1}\cap W^{l_{0}+1}_{r_{0}+1} \right) \subset \subset X \subset \subset\bigcup\limits_{1 \leq i \leq l_{0}} \left( W^{l_{0}}_{i}\cap W_{l_{0}+1}^{l_{0}}\cap W^{l_{0}}_{r_{0}+1} \right).$$ 
	By Lemma \ref{contextension}, we find a map  $\tilde{h}_{(l_{0}+1)(r_{0}+1)} \in W^{k,p}\left( V^{r_{0}}_{l_{0}+1}\cap V^{r_{0}}_{r_{0}+1}; G\right) $
	such that $ \tilde{h}_{(l_{0}+1)(r_{0}+1)} = \hat{h}_{(l_{0}+1)(r_{0}+1)} $
	in a neighborhood of $X$ and we have
	\begin{multline*}
	\left\lVert \operatorname{exp}^{-1} \left(  g_{(l_{0}+1)(r_{0}+1)}^{-1}\tilde{h}_{(l_{0}+1)(r_{0}+1)}\right)  \right\rVert_{W^{k,p}\left( V^{r_{0}}_{l_{0}+1}\cap V^{r_{0}}_{r_{0}+1}; \mathfrak{g}\right)}  \\ \lesssim \left\lVert \operatorname{exp}^{-1} \left(  g_{(l_{0}+1)(r_{0}+1)}^{-1}\hat{h}_{(l_{0}+1)(r_{0}+1)}\right) \right\rVert_{W^{k,p}\left( \bigcup\limits_{1 \leq i \leq l_{0}} \left( W^{l_{0}}_{i}\cap W_{l_{0}+1}^{l_{0}}\cap W^{l_{0}}_{r_{0}+1} \right); \mathfrak{g}\right)}. 
	\end{multline*}  
	Combining this estimate with \eqref{small_osc} and choosing $\varepsilon $ smaller if necessary, we can force that $ \operatorname{osc}\left(\tilde{h}_{(l_{0}+1)(r_{0}+1)}; W_{l_{0}+1}^{l_{0}}\cap W^{l_{0}}_{r_{0}+1} \right)\leq \delta_{G}.$ Hence by Lemma \ref{smoothinglemma}, there exists $h^{l_{0}+1}_{(l_{0}+1)(r_{0}+1)} \in C^{\infty}\left( W^{l_{0}+1}_{l_{0}+1}\cap W^{l_{0}+1}_{r_{0}+1} ; G\right)$ satisfying \eqref{cocycle_l+1} and \eqref{estimate_l+1}. \end{proof}

\end{appendices}
\section*{Acknowledgments}
The author warmly thanks Tristan Rivi\`{e}re for introducing him to this subject and numerous discussions, suggestions and encouragement.

\end{document}